\documentclass[11pt]{amsart}
\usepackage{amssymb,latexsym,amsmath,longtable,mathdots}
\usepackage[mathscr]{eucal}
\usepackage{color,lscape}

\numberwithin{equation}{section}

\newcommand{\hsp}[1]{{\hbox{\hspace{#1}}}}

\newcommand{\mystack}[2]{\ensuremath{ \substack{ \hbox{\scriptsize{${#1}$}} \\ \hbox{\scriptsize{${#2}$}} }} }

\def\a{\alpha}  
\def\b{\beta}  
\def\c{\gamma}  
\def\d{\delta}  
\def\e{\varepsilon}  

\def\z{\zeta}
\def\m{\mu}
\def\n{\nu}
\def\s{\sigma}

\def\x{\xi}

 \def\sfa{\mathsf{a}}
  
\def\tAd{\mathrm{Ad}}

\def\fb{\mathfrak{b}} 

\def\bC{\mathbb C}

\def\tcodim{\mathrm{codim}}

\def\tdiag{\mathrm{diag}} \def\tdim{\mathrm{dim}}

\def\tEnd{\mathrm{End}}

\def\texp{\mathrm{exp}}

\def\tGr{\mathrm{Gr}}
\def\fg{{\mathfrak{g}}} \def\tfg{\tilde\fg}

\def\fh{\mathfrak{h}}

\def\tti{\mathtt{i}} 

\def\bI{\mathbf{I}}

 \def\ttj{\mathtt{j}} 
  
\def\ttJ{\mathtt{J}}

 \def\ttk{\mathtt{k}}

\def\tLG{\mathrm{LG}}

\def\tmax{\mathrm{max}}

\def\fn{\mathfrak{n}}

\def\bP{\mathbb P}

\def\fp{\mathfrak{p}} 
\def\sfp{\mathsf{p}}

 \def\cQ{\mathcal Q} 
 
\def\fq{\mathfrak{q}} 
\def\sfq{\mathsf{q}}

\def\fr{\mathfrak{r}}
\def\sfr{\mathsf{r}}

 \def\cS{\mathcal S}

\def\fs{\mathfrak{s}}
 \def\sfs{\mathsf{s}}
\def\tSL{\mathrm{SL}} 
 \def\tSpin{\mathit{Spin}}
 \def\tSing{\mathrm{Sing}}
\def\tStab{\mathrm{Stab}}

\def\fsl{\mathfrak{sl}}

\def\sft{\mathsf{t}}

 \def\bZ{\mathbb Z}

\def\half{\tfrac{1}{2}}

\def\dfn{\stackrel{\hbox{\tiny{dfn}}}{=}}
\def\sbullet{{\hbox{\tiny{$\bullet$}}}}

\def\op{\oplus}
\def\ot{\otimes}

\def\wt{\widetilde}

\def\tw{\hbox{\small $\bigwedge$}}

\newcounter{numcnt}

\newcounter{cnt}

\newcounter{acnt}
\newenvironment{a_list}{ 
  \begin{list}{{(\alph{acnt})}}
   {\usecounter{acnt} \setlength{\itemsep}{3pt}
    \setlength{\leftmargin}{25pt} \setlength{\labelwidth}{20pt} }
   }
   {\end{list}}
\newenvironment{a_list_emph}{ 
  \begin{list}{{\emph{(\alph{acnt})}}}
   {\usecounter{acnt} \setlength{\itemsep}{3pt}
    \setlength{\leftmargin}{25pt} \setlength{\labelwidth}{20pt} }
   }
   {\end{list}}

\newcounter{Acnt}

\newcounter{icnt}

\newcounter{Icnt}

\newcounter{exam_cnt}

\newcounter{mccnt}
\newenvironment{circlist}{ 
  \begin{list}{$\circ$}
   {\usecounter{cnt} \setlength{\itemsep}{2pt}
    \setlength{\leftmargin}{15pt} \setlength{\labelwidth}{20pt} }
   }
   {\end{list}}

\newtheorem{corollary}[equation]{Corollary}
\newtheorem*{corollary*}{Corollary}
\newtheorem{lemma}[equation]{Lemma}
\newtheorem*{lemma*}{Lemma}
\newtheorem{proposition}[equation]{Proposition}
\newtheorem*{proposition*}{Proposition}
\newtheorem{theorem}[equation]{Theorem}
\newtheorem*{theorem*}{Theorem}

\theoremstyle{definition}
\newtheorem*{question*}{Question}          

\newtheorem*{boldQ*}{Question}

\theoremstyle{remark}
\newtheorem*{assume*}{Assume}
\newtheorem*{answer*}{Answer}

\newtheorem*{claim*}{Claim}

\newtheorem*{definition*}{Definition}
\newtheorem{example}[equation]{Example}
\newtheorem*{example*}{Example}
\newtheorem*{hint*}{Hint}
\newtheorem*{notation*}{Notation}
\newtheorem{remark}[equation]{Remark}
\newtheorem*{remark*}{Remark}
\newtheorem*{remarks*}{Remarks}
\newtheorem*{fact*}{Fact}
\newtheorem*{emphL*}{Lemma}

\newtheorem*{emphQ*}{Question}

\numberwithin{HWeq}{section}
\theoremstyle{definition}

\def\tfg{\tilde{\mathfrak{g}}}
\begin{document}
\title{Singular loci of cominuscule Schubert varieties}
\author[Robles]{C. Robles}
\email{robles@math.tamu.edu}
\address{Mathematics Department, Mail-stop 3368, Texas A\&M University, College Station, TX  77843-3368} 
\thanks{Robles is partially supported by NSF DMS-1006353.}
\date{\today}
\begin{abstract}
Let $X = G/P$ be a cominuscule rational homogeneous variety.  Equivalently, $X$ admits the structure of a compact Hermitian symmetric space.  I give a uniform description (that is, independent of type) of the irreducible components of the singular locus of a Schubert variety $Y\subset X$ in terms of representation theoretic data.  The result is based on a recent characterization of the Schubert varieties by an integer $\sfa\ge0$ and a marked Dynkin diagram.  Corollaries include: (1) the variety is smooth if and only if $\sfa=0$;  (2) if $G$ of Type ADE, then the singular locus occurs in codimension at least three.
\end{abstract}
\maketitle

\setcounter{tocdepth}{1}

\section{Introduction}

Let $X=G/P$ be a cominuscule rational homogeneous variety.  (Equivalently, $X$ admits the structure of a compact Hermitian symmetric space.  An example is the Grassmannian $\tGr(k,n)$ of $k$-planes in complex $n$-space.)   The main result (Theorem \ref{T:sing}) of this paper is a uniform (independent of $G$) description of the irreducible components of the singular loci of Schubert varieties $Y \subset X$.

\subsection*{Context and related results}

It is an important open question, to explicitly describe the singular loci of Schubert varieties in an arbitrary (not necessarily cominuscule) rational homogeneous variety $X = G/P$.  The problem lies at the interface between geometry, combinatorics and representation theory, and has stimulated research in each area.  There is a vast body of literature on the subject, and the interested reader may consult \cite{MR1782635} for an excellent overview.

In general, characterizations of smooth Schubert varieties, and descriptions of singular loci tend to be combinatorial in nature.  As an example, consider the full flag variety $X= \tSL_n\bC/B$, where $B$ is a Borel subgroup (e.g., upper triangular matrices).  In this case, the Schubert varieties are indexed by the symmetric group $S_n$ on $n$ letters.  Given $w \in S_n$, V. Lakshmibai and B. Sandhya \cite{MR1051089} showed that the corresponding Schubert variety $X_w$ is smooth if and only if the permutation avoids the patterns $(3412)$ and $(4231)$.  At that time, they gave a conjectural (and combinatorial) description of the irreducible components (which are necessarily Schubert varieties) of singular locus.  Gasharov \cite{MR1827861} established the sufficiency of the Lakshmibai--Sandhya conditions, and the full conjecture was established (independently) by S. Billey and G. Warrington \cite{MR1990570}, A. Cortez \cite{MR1994224}, C. Kassel, A. Lascoux and C. Reutenaur \cite{MR2015302}, and L. Manivel \cite{MR1853139}.  More recently, Billey and Postnikov \cite{MR2123545} have uniformly extended the Lakshmibai--Sandhya smoothness criterion to generalized flag manifolds $X = G/B$.  An advantage of the pattern avoidance criteria over other characterizations (see \cite{MR1782635}) of smoothness (and the weaker rational smoothness) is that it is nonrecursive and computationally efficient.

Returning to the case that $X = G/P$ is cominuscule (in general, $P$ is not a Borel subgroup), type--specific descriptions in the case that $X$ is classical ($G$ is of type ABCD) have been known for some time:  the type A case ($X$ is a Grassmannian) since the 1970s, and the rest by 1990; see \cite[Section 9.3]{MR1782635} and the references therein.  Those descriptions are given in terms of partitions.  In contrast, the  characterization of Theorem \ref{T:sing} is given by representation theoretic data, and so presents a complimentary perspective.  It also has the advantage of being independent of type, and yields an explicit description of the singular loci in the two non-classical cases ($G = E_6,E_7$).  Like the pattern advoidence criteria discussed above, the characterization is both nonrecursive and amenable to compuation.  

The cominuscule $X$ are closely related to the minuscule rational homogeneous varieties.  Indeed, with the exception of the quadric hypersurface $\cQ^{2n-1} = B_n/P_1$ and the Lagrangian Grassmannian $\tLG(n,2n) = C_n/P_n$, every irreducible cominuscule $X$ is minuscule.  Theorem \ref{T:sing} compliments descriptions of the singular loci of Schubert varieties in minuscule $X$ by M. Brion and P. Polo \cite{MR1703350}, and by N. Perrin \cite{MR2466424}.  (Brion and Polo also study singularities of cominuscule Schubert varieties, but -- so far as I can discern -- stop short of a complete description of the irreducible components.)

\subsection*{Contents}

In Section \ref{S:review} we review Schubert varieties, and their representation theoretic characterization by an integer $\sfa\ge0$ and a marking $\ttJ$ of the Dynkin diagram of $G$.  (The pair $(\sfa,\ttJ)$ encodes the relationship between $P$ and the stabilizer of the Schubert variety; see Remark \ref{R:Gw}.)  The main theorem and subsequent corollaries are discussed in Section \ref{S:sing}.  Corollaries \ref{C:|sing|} and \ref{C:sing} describe the relationship between the integer $\sfa = \sfa(Y)$ and the number of irreducible components in $\tSing(Y)$.  For example, $Y$ is smooth if and only if $\sfa=0$; if $\sfa=1$, then $\tSing(Y)$ is irreducible.  Corollary \ref{C:scd} gives lower bounds on the codimension of the singular locus, and characterizes those Schubert varieties for which the bound is realized.  The main result is proved in Section \ref{S:prf}.

\tableofcontents

\subsection{Use of LiE} \label{S:lie}

There are two exceptional, irreducible compact Hermitian symmetric spaces: the Cayley plane $E_6/P_6$ and the Freudenthal variety $E_7/P_7$.  The software \cite{LiE} is used to perform computations and verify some results for these two cominuscule varieties  (Section \ref{S:singE} and Appendix \ref{S:aJE}).  I emphasize that none of the proofs of the paper rely on LiE; the software is used only to compute the $(\sfa,\ttJ)$--values for these two exceptional varieties.  Moreover, the code I wrote applies to any cominuscule variety; that is, it is \emph{not} specialized to the exceptional cases.  This allowed me to test the code by applying it to classical cominuscule varieties (of low rank -- the code requires that the rank of $G$ be specified).   The outputs are consistent with the ``by hand" results obtained for the classical cases.  This provides some confidence for the accuracy of the code and the corresponding computations in the two exceptional cases.

\subsection*{Acknowledgements}
Over the course of this project, I have benefitted from conversations and/or correspondence with many people, including S. Billey, J. Carrell, V. Lakshmibai, F. Sottile, D. The, A. Woo and A. Yong.  I thank them for their insights and time.

\section{Review} \label{S:review}

\subsection{Notation and background} \label{S:not}

The present article is founded on a result of \cite{MR2960030}.  With the exception noted in Remark \ref{R:newaJ}, I will use the notation of that paper.  To streamline the presentation, I will give a laconic review of the discussion of rational homogeneous varieties, their Schubert subvarieties, grading elements and Hasse diagrams in \cite[Sections 2.1-2.4 and 3.1]{MR2960030}.  Briefly, $G$ is a complex simple Lie group.  A choice of Cartan and Borel subgroups $H \subset B$ has been fixed, $P \supset B$ is a maximal parabolic subgroup associated with a cominuscule root, and $X = G/P$ is the corresponding cominuscule variety.  The associated Lie algebras are denoted $\fh \subset \fb \subset \fp \subset \fg$.  Let $W$ denote the Weyl group of $\fg$, and $W_\fp$ the Weyl group of the reductive component in the Levi decomposition of $\fp$.  The Hasse diagram $W^\fp$ is the set of minimal length representatives of the right--coset space $W_\fp \backslash W$, and indexes the Schubert classes.  Let 
$$
  o = P/P \ \in \ X = G/P \, .
$$
Given $w \in W^\fp$, the Zariski closure
$$
  X_w \ := \ \overline{B w^{-1} \cdot o}
$$
is a Schubert variety.  Any $G$--translate of the Schubert variety $X_w$ will be referred to as a \emph{Schubert variety of type $w$}.  Let $\xi_w = [X_w] \in H_{2|w|}(X,\bZ)$ denote the corresponding Schubert class. 

Let $\{ Z_1,\ldots,Z_r\}$ be the basis of $\fh$ dual to the simple roots $\Sigma=\{\a_1,\ldots,\a_r\}$.  Let $\a_\tti$ be the simple root associated with the cominuscule $\fp$.  Since $Z_\tti$ is an element of the Cartan subalgebra $\fh$, the vector space $\fg$ decomposes into a direct sum of $Z_\tti$--eigenspaces.  The homogeneous variety $G/P$ is cominuscule if and only if the eigenvalues are $\{ -1,0,1\}$; that is, \begin{equation} \label{E:gi}
  \fg \ = \ \fg_1 \,\op\, \fg_0 \,\op\, \fg_{-1} \,,\quad\hbox{where}\quad
  \fg_k \, := \, \{ A \in \fg \ | \ [Z_\tti \,,\, A] = k A \} \,.
\end{equation}
The decomposition \eqref{E:gi} is the \emph{$Z_\tti$--graded decomposition} of the Lie algebra $\fg$.  Moreover,
\begin{equation} \label{E:p}
  \fp \ = \ \fg_1 \,\op\, \fg_0 \,,
\end{equation}
and $\fg_0$ is the reductive component of the parabolic subalgebra $\fp$.

\begin{remark} \label{R:abelian}
As a graded decomposition, we have $[\fg_k \,,\, \fg_\ell ] \subset \fg_{k+\ell}$.  In particular, the subspaces $\fg_{\pm1}$ are both $\fg_0$--modules, and abelian subalgebras of $\fg$.
\end{remark}

\noindent Equation \eqref{E:p} implies 
$$
  T_oX \ \simeq \ \fg_{-1}
$$
as an $\fg_0$--module.

\subsection*{Notation} 
Let $\Delta$ denote the set of roots of $\fg$.  Given $\a \in \Delta$, let $\fg_\a \subset \fg$ denote the corresponding root space.  Given any subset $\fs \subset \fg$, let 
$$
  \Delta(\fs) \ = \ \{ \a \in \Delta \ | \ \fg_\a \subset \fs \} \,.
$$
Given a subset $U$ of a vector space, let $\langle U \rangle$ denote the linear span.

\subsection{The characterization of Schubert varieties} \label{S:aJ}

This section is a concise review of the characterization of Schubert varieties $\xi_w$ by an integer $\sfa(w)\ge0$ and a marking $\ttJ(w)$ of the Dynkin diagram.  (The marking is equivalent to a choice of simple roots from $\Sigma \backslash\{\a_\tti\}$.)  For more detail see \cite{MR2960030}.  Given $w \in W^\fp$, define
\begin{equation} \label{E:nw}
  \Delta(w) \ = \ w\Delta^-\cap\Delta^+ \ \subset \ \Delta(\fg_1)
  \quad\hbox{and}\quad 
  \fn_w \ = \ \bigoplus_{\a\in\Delta(w)} \, \fg_{-\a} \ \subset \ \fg_{-1} \,.
\end{equation}
The set $\Delta(w)$ is known as the \emph{inversion set of $w$}.  Let $N_w = \texp(\fn_w)$.  
Then
\begin{equation}\label{E:Xw}
  Y_w \ := \ \overline{N_w \cdot o} \ = \ w X_w
\end{equation}
is a Schubert variety of type $w$. 

By work of Kostant \cite{MR0142696}, the $\ell$--th exterior power $\tw^\ell\fg_{-1}$ decomposes into irreducible $\fg_0$--modules $\bI_w$, which are indexed by elements of $W^\fp$ of length $\ell$
$$
  \tw^\ell \fg_{-1} \ = \ \bigoplus_{\mystack{w\in W^\fp}{|w|=\ell}} \bI_w \,.
$$
The highest weight line in $\bP \bI_w$ is $\fn_w$.   Let $1 \in W^\fp$ be the identity, and let $w_0 \in W^\fp$ be the longest element.  Then $Y_1 = o$ and $Y_{w_0} = X$, and $\bI_1$ and $\bI_{w_0}$ are trivial $\fg_0$--modules.  Assume $w \in W^\fp \backslash \{1,w_0\}$.  Let $\fq_w \subset \fg_0$ be the stabilizer of the highest weight line $\fn_w$.  Then there is a subset $\ttJ(w) \subset \{ 1 , \ldots , r \} \backslash\{\tti\}$ with the property that the Lie algebra $\fq_w$ is given by $\fq_w = \fg_{0,\ge0}$, where 
\begin{equation} \label{E:gkl}
  \fg_{k,\ell} \ := \ \{ A \in \fg_k \ | \ [Z_w \,,\, A] = \ell A \} 
  \quad \hbox{ and } \quad Z_w := \sum_{\ttj\in \ttJ(w)} Z_\ttj \, .
\end{equation}
We call $\fg = \oplus \fg_{k,\ell}$ the \emph{$(Z_\tti,Z_w)$--bigraded decomposition} of $\fg$.  It is a simple consequence of standard representation theory that $k\ell < 0$ forces $\fg_{k,\ell}=0$.  The following is \cite[Proposition 3.9]{MR2960030}.

\begin{proposition}[{\cite{MR2960030}}]\label{P:aJ}
Let $w \in W^\fp\backslash\{1,w_0\}$.  There exists an integer $\sfa = \sfa(w) \ge 0$ such that the inversion set is given by
\begin{subequations} 
\begin{equation} \label{E:Dw}
  \Delta(w) \ = \ \{ \a \in \Delta(\fg_1) \ | \ \a(Z_w) \le \sfa \}\,.
\end{equation}
Equivalently,
\begin{equation} \label{E:nw_aJ}
  \fn_w \ = \ \fg_{-1,0} \ \op \ \cdots \ \op \ \fg_{-1,-\sfa} \, .
\end{equation}
\end{subequations}
\end{proposition}

\begin{remark}  
\begin{a_list}
\item
By \cite[Proposition 3.19]{MR2960030}, the Schubert variety $Y_w$ is smooth if and only if $\sfa(w)=0$.  In particular, the proposition generalizes the characterization of the smooth Schubert varieties of $X$ by connected Dynkin sub-diagrams containing the $\tti$--th node.  For more on the relationship between the integer $\sfa(w)$ and $\tSing(Y_w)$, see Corollary \ref{C:sing}.  
\item 
By \eqref{E:nw_aJ}, the pair $\sfa(w),\ttJ(w)$ characterizes $\fn_w$ as a direct sum of $Z_w$--eigenspaces.  This is the computational tool that we will use to work with the inversion set $\Delta(w)$.
\item
In the case that $X$ is a classical (types ABCD), well-known descriptions of Schubert varieties are given by partitions.  Appendix \ref{S:app1} provides a `translation' between the $(\sfa,\ttJ)$ and partition descriptions.
\item
The $(\sfa,\ttJ)$--values for the Schubert classes in the two non-classical cominuscule varieties are given by Figures \ref{f:E6} and \ref{f:E7}.   
\item
Since $\xi_w = [Y_w]$, and $Y_w$ is determined by $\Delta(w)$, the pair $\sfa(w), \ttJ(w)$ characterizes $\xi_w$, when $w \in W^\fp\backslash\{1,w_0\}$.  
\item
A tableau-esque analog of Proposition \ref{P:aJ} is given by H. Thomas and A. Yong in \cite[Proposition 2.1]{MR2538022}.
\end{a_list}
\end{remark}

\begin{remark}\label{R:newaJ}
I follow the notation of \cite{MR2960030}, with the following exception.  In \cite{MR2960030}, we uniformly write $\ttJ = \{\ttj_1<\cdots<\ttj_\sfp\}$.  Here, it is convenient to reorder the $\ttj_\ell$ in some cases.
\end{remark}

\section{Singular locus} \label{S:sing}

For this section we fix, once and for all, $w \in W^\fp\backslash\{1,w_0\}$.  The singular locus $\tSing(X_w)$ is a union of Schubert subvarieties $X_{w'} \subset X_w$.  Let $\tSing_w \subset W^\fp$ be the subset indexing the irreducible components of $\tSing(X_w)$, so that 
$$
  \tSing(X_w) \ = \ \bigcup_{w' \in \tSing_w}  X_{w'} \,.
$$

\begin{definition*}
Let $\Pi_{1,\sfa-1} := \{ \e \in \Delta(\fg_{1,\sfa-1}) \ | \ \e + \a \not\in\Delta \ \forall \ \a\in\Delta^+(\fg_{0,0})\}$.  Equivalently, $\Pi_{1,\sfa-1}$ is the set of highest weights associated with the $\fg_{0,0}$--module $\fg_{1,\sfa-1}$.
\end{definition*}

Let $\e \in \Pi_{1,\sfa-1}$.  Define 
\begin{equation} \label{E:we}
\renewcommand{\arraystretch}{1.2} \begin{array}{rcl}
  \Delta(w,\e) & := & \{\e\} \ \sqcup \ 
  \{ \n \in \Delta(\fg_{1,\sfa}) \ | \ \n-\e \in \Delta(\fg_{0,1}) \} \\
  & = & 
  \{\e\} \ \sqcup \ \left\{ \Delta \cap 
  \left(\e+\Delta(\fg_{0,1})\right) \right\} \, .
\end{array}\end{equation}

\begin{lemma} \label{L:we}
There exists $w_\e \in W^\fp$ such that $\Delta(w_\e) = \Delta(w) \backslash \Delta(w,\e)$.
\end{lemma}

\begin{theorem}  \label{T:sing}
The roots $\Pi_{1,\sfa-1}$ are in bijective correspondence with the irreducible components of the singular locus $\tSing(X_w)$.  Explicitly, 
$$
 \tSing_w \ = \ \{ w_\e \ | \ \e \in \Pi_{1,\sfa-1} \} \, .
$$
\end{theorem}

\noindent The lemma and theorem are proved in Section \ref{S:lemmas}.

\begin{example*} 
Theorem \ref{T:sing} generalizes the well-known descriptions \cite[Section 9.3]{MR1782635} of the singular locus of $X_w$ in the case that $G$ is classical.  Here is a simple illustration in the case that $X = \tGr(5,11)$ is a Grassmannian.  Assume the notations and definitions of Section \ref{S:aJA}.  In particular, the subalgebra $\fg_{-1} \subset \fsl_n\bC$ is spanned by $\{ e^j_k \ | \ 1 \le j \le 5 \,, \ 6 \le k \le 11\}$.  The basis element $e^j_{k+1} \in \fg_{-1}$ is a root vector for the root
$$
  -\a_{jk} \ := \ - \left( \a_j + \cdots + \a_k \right) \,. 
$$

Fix $w \in W^\fp$ with $\sfa(w)=2$ and $\ttJ(w) = \{ 2 \,,\, 3 \,,\, 6 \,,\, 8 \,,\, 10 \}$.  
Any element $s^k_j e_k^j$ of $\fg_{-1}$ may be represented by a matrix $(s^j_k)$ where $1 \le j \le 5$ and $6 < k \le 11$.  The corresponding $Z_w$--degrees are given by
\begin{small}
$$
-\left[\begin{array}{cc|c|cc}
  2 & 2 & 1 & 0 & 0 \\ \hline 
  3 & 3 & 2 & 1 & 1 \\ 3 & 3 & 2 & 1 & 1 \\ \hline
  4 & 4 & 3 & 2 & 2 \\ 4 & 4 & 3 & 2 & 2 \\ \hline
  5 & 5 & 4 & 3 & 3 
\end{array}\right] \,.
$$
\end{small}
So the subspace $\fn_w \subset \fg_{-1}$ is represented by  
\begin{small}
$$ \renewcommand{\arraystretch}{1.2}
  \fn_w \ = \ \left( \begin{array}{cc|c|cc} 
    s^6_1 & s^6_2 & s^6_3 & s^6_4 & s^6_5 \\ \hline
    0 & 0 & s^7_3 & s^7_4 & s^7_5 \\ 
    0 & 0 & s^8_3 & s^8_4 & s^8_5 \\ \hline
    0 & 0 & 0 & s^9_4 & s^9_5 \\ 
    0 & 0 & 0 & s^{10}_4 & s^{10}_5 \\ \hline
    0 & 0 & 0 & 0 & 0
  \end{array} \right) \, .
$$
\end{small}
Define a filtration $F^3 \subset F^6 \subset F^{10} \subset \bC^{11}$ by $F^3 = \langle e_1 , e_2 , e_6 \rangle$, $F^6 = \langle F^3 , e_3 , e_7 , e_8 \rangle$ and $F^{10} = \langle F^6 , e_4 , e_5 , e_9 , e_{10} \rangle$.  Then \eqref{E:Xw} yields $X_w = w^{-1} Y_w$, with
$$
  Y_w \ = \ \{ E \in \tGr(5,11) \ | \ \tdim(E \cap F^3) \ge 2 \,, \ 
  \tdim(E \cap F^6) \ge 3 \,,\ \tdim(E\cap F^{10}) \ge 5 \} \,.
$$

The subalgebra $\fg_{0,0}$ is identified with the diagonal block matrices $\tdiag(2,1,2,1,2,2,1)$ in $\fsl_{11}$.  The $\fg_{0,0}$--module $\fg_{-1,-1}$ consists of two irreducible submodules.  The first is spanned by $e^3_6$, with highest weight $\e_1 = -\a_{35}$.  The second is four-dimensional with highest weight vector $e^4_8$ and highest weight $\e_2 = -\a_{47}$.  We have $\Delta(w,\e_1) = \{ \a_{15}\,,\, \a_{25}\,,\,\a_{35} \,,\, \a_{36}\,,\,\a_{37}\}$ and $\Delta(w,\e_2) = \{ \a_{37}\,,\,\a_{47}\,,\,\a_{48}\,,\,\a_{49} \}$.\footnote{Note that the associated subspaces of $\fn_w$ are `hooks' that are added to $\pi$ to obtain the partitions associated with the irreducible components of $\tSing(X_\pi)$.}  The two irreducible components $X_{w_1} = w_1^{-1}Y_{w_1}$ and $X_{w_2} = w_2^{-1} Y_{w_2}$ of $\tSing(X_w)$ correspond to 
\begin{small}
$$ \renewcommand{\arraystretch}{1.2}
  \fn_{w_1} \ = \ \left( \begin{array}{cc|c|cc} 
    0 & 0 & 0 & s^6_4 & s^6_5 \\ \hline
    0 & 0 & 0 & s^7_4 & s^7_5 \\
    0 & 0 & 0 & s^8_4 & s^8_5 \\ \hline
    0 & 0 & 0 & s^9_4 & s^9_5 \\
    0 & 0 & 0 & s^{10}_4 & s^{10}_5 \\ \hline
    0 & 0 & 0 & 0 & 0
  \end{array} \right)
  \quad\hbox{ and } \quad
  \fn_{w_2} \ = \ \left( \begin{array}{cc|c|cc} 
    s^6_1 & s^6_2 & s^6_3 & s^6_4 & s^6_5 \\ \hline
    0 & 0 & s^7_3 & s^7_4 & s^7_5 \\
    0 & 0 & 0 & s^8_4 & s^8_5 \\ \hline
    0 & 0 & 0 & 0 & s^9_5 \\
    0 & 0 & 0 & 0 & s^{10}_5 \\ \hline
    0 & 0 & 0 & 0 & 0
  \end{array} \right) \,.
$$
\end{small}
The corresponding $\sfa$ and $\ttJ$ values are $\sfa(w_1) = 0$ and $\ttJ(w_2) = \{ 3 \,,\, 10 \}$; and $\sfa(w_2) = 2$ and $\ttJ_{w_2} = \{ 2\,,\,4\,,\, 6\,,\,7\,,\,10\}$.
\end{example*}

\begin{corollary} \label{C:a=1}
The Schubert variety $X_w$ is smooth if and only if $\sfa(w)=0$.  If $\sfa(w)=1$, then $\tSing(X_w)$ is a single irreducible Schubert variety.
\end{corollary}

\begin{remark*}
The first part of the corollary was proved in \cite[Proposition 3.9]{MR2960030}.
\end{remark*}

\begin{proof}
The first statement is an immediate consequence of $\fg_{1,\sfa-1} = \fg_{1,-1} = \{0\}$.  The second statement is a consequence of the fact that $\fg_{1,\sfa-1} = \fg_{1,0}$ is an irreducible $\fg_{0,0}$--module, see \cite[Section 3.2]{MR2960030}.
\end{proof}

The following is an immediate consequence of Theorem \ref{T:sing}.

\begin{corollary} \label{C:|sing|}
The number of irreducible components in $\tSing(X_w)$ is the number $|\Pi_{1,\sfa-1}|$ of components in a decomposition of $\fg_{1,\sfa-1}$ into irreducible $\fg_{0,0}$--submodules.
\end{corollary}

\noindent From Corollary \ref{C:|sing|}, and the $(\sfa,\ttJ)$--characterizations of Sections \ref{S:aJA}--\ref{S:aJD} we deduce 

\begin{corollary} \label{C:sing}
Let $X_w$ be a Schubert variety in a classical, irreducible cominuscule $X = G/P$, and let $|\tSing_w|$ be the number of irreducible components in $\tSing(X_w)$.
\begin{a_list_emph}
  \item If $X = \tGr(\tti,n+1)$, then $|\tSing_w| = \sfa(w)$.
  \item If $X = \tLG(n,2n)$, then $|\tSing_w| = \lceil \sfa(w)/2 \rceil$.
  \item Suppose $X = \cS_n$ and assume $\sfa(w) > 1$.  Set $\sfr = \lceil \half (\sfa+\a_{n-1}(Z_w)) \rceil$.  If $1 = \ttj_{\sfr-1} - \ttj_{\sfr}$, then $|\tSing_w| = \lfloor \half (\sfa+\a_{n-1}(Z_w))\rfloor \in\{\sfr-1,\sfr\}$; 
         otherwise $|\tSing_w| = \sfr$.
\end{a_list_emph}
\end{corollary}

\begin{remark*}
The singular loci of Schubert varieties in quadric hypersurfaces $B_n/P_1$ and $D_n/P_1$ are so simple that I omitted them from the corollary.  See \cite[Section 9.3]{MR1782635}.
\end{remark*}

\begin{proof}[Proof of Corollary \ref{C:sing}(a)]
Adopt the notation of Appendix \ref{S:aJA}, and assume $\sfa>0$.  Then the highest $\fg_{0,0}$--weights of $\fg_{1,\sfa-1}$ are 
$$
  \e \ = \ \a_{\ttj_\ell+1} + \cdots + \a_\tti + \cdots + \a_{\ttk_m-1} \,,
$$
with $0 < \ell,m$ and $\ell+m = \sfa+1$.  Thus $|\Pi_{1,\sfa-1}| = \sfa(w)$.
\end{proof}

\begin{proof}[Proof of Corollary \ref{C:sing}(b)]
Adopt the notation of Appendix \ref{S:aJA}, and assume $\sfa>0$.  Suppose that $\sfa = 2\sfs$.  Then the highest $\fg_{0,0}$--weights of $\fg_{1,\sfa-1}$ are 
\begin{equation} \label{E:e} \tag{$\ast$}
  \e \ = \ \a_{\ttj_m+1} + \cdots + \a_{\ttj_\ell} 
            + 2 (\a_{\ttj_\ell+1} + \cdots + \a_{n-1}) + \a_n \, , 
\end{equation}
with $1 \le \ell < m$ and $\ell + m = \sfa+1 = 2\sfs+1$.  Therefore, $|\Pi_{1,\sfa-1}| = \sfs = \lceil \sfa(w)/2 \rceil$.

If $\sfa = 2\sfs-1$, then the highest $\fg_{0,0}$--weights of $\fg_{1,\sfa-1}$ are $(\ast)$ and 
$$
  \e \ = \ 2 (\a_{\ttj_\sfs+1} + \cdots + \a_{n-1}) + \a_n \, .
$$
Thus, $|\Pi_{1,\sfa-1}| = \sfs = \lceil \sfa(w)/2 \rceil$.
\end{proof}

\noindent The proof of Corollary \ref{C:scd}(c), which is very similar to, though more tedious than, that of Corollary \ref{C:scd}(b), is left to the reader.

\subsection{The exceptional cases} \label{S:singE}

The irreducible components of $\tSing(X_w)$ have been determined in the case that $G$ is classical and $P$ is (co)minuscule; see \cite[Section 9.3]{MR1782635}.  A fourth corollary of Theorem \ref{T:sing} is an explicit description of the singular locus of the Schubert varieties in the exceptional Cayley plane $E_6/P_6$ and Freudenthal variety $E_7/P_7$ (both of which are minuscule and cominuscule).  See Tables \ref{t:E6} and \ref{t:E7} on pages \pageref{t:E6} and \pageref{t:E7}, respectively.  The tables are obtained with the assistance of \cite{LiE}.

\subsubsection*{Key to Tables \ref{t:E6} and \ref{t:E7}}  
Each row represents a proper ($\not= o , X$) Schubert variety $X_w$ of $X$, indexed by $w\in W^\fp\backslash\{1,w_0\}$.  The first column is the dimension of $X_w$; the second column expresses $w$ as a reduced product of simple reflections, acting on the left; the third column gives the corresponding $\sfa(w):\ttJ(w)$ values, see Section \ref{S:aJ}; and the fourth column lists the irreducible components of the singular locus in terms of their $\sfa:\ttJ$ characterization.  (See also Figures \ref{f:E6} and \ref{f:E7} on pages \pageref{f:E6} and \pageref{f:E7}, respectively.)

\begin{remarks*}
From the tables, we see that:
\begin{a_list}
\item The the irreducible components $\{ X_{w_\e} \ | \ \e\in\Pi_{1,\sfa-1} \}$ of the singular locus of a Schubert variety $X_w$ in $E_6/P_6$ or $E_7/P_7$ satisfy $\tcodim_{X_w}X_{w_\e} \ge 3$.
\item The singular locus of a Schubert variety in the Cayley plane consists of at most one irreducible component.
\item The singular locus of a Schubert variety in the Freudenthal variety consists of at most two irreducible components.
\end{a_list}
\end{remarks*}

\noindent
\begin{table}[htb]
\caption{Schubert varieties and their singular loci in the Cayley plane.}
\label{t:E6}
\renewcommand{\arraystretch}{1.2}
\noindent
\begin{tabular}{|c|l|l|l|}
\hline
dim & $w$ & $\sfa:\ttJ$ & $\tSing_w$ \\ \hline\hline
1 & 6 & 0:5 &  \\ \hline
2 & 65 & 0:4 &  \\ \hline
3 & 654 & 0:23 & \\ \hline
4 & 6542 & 0:3 & \\ \hline
4 & 6543 & 0:12 & \\ \hline
5 & 65432 & 1:123 & 0:4  \\ \hline
5 & 65431 & 0:2 & \\ \hline
6 & 654321 & 1:23 & 0:4  \\ \hline
6 & 654324 & 1:14 & 0:5  \\ \hline
7 & 6543241 & 2:124 & 0:3  \\ \hline
7 & 6543245 & 1:15 & $o$  \\ \hline
8 & 65432413 & 1:4 & 0:5  \\ \hline
8 & 65432451 & 2:125 & 0:3  \\ \hline
\end{tabular}
\hspace{10pt}
\begin{tabular}{|c|l|l|l|}
\hline
dim & $w$ & $\sfa:\ttJ$ & $\tSing_w$ \\ \hline\hline
8 & 65432456 & 0:1 & \\ \hline
9 & 654324513 & 3:145 & 1:23  \\ \hline
9 & 654324561 & 1:12 & 0:3  \\ \hline
10 & 6543245134 & 2:35 & 0:2  \\ \hline
10 & 6543245613 & 2:14 & 1:23  \\ \hline
11 & 65432451342 & 1:5 & $o$  \\ \hline
11 & 65432456134 & 3:135 & 1:4  \\ \hline
12 & 654324561342 & 2:15 & 1:4  \\ \hline
12 & 654324561345 & 1:3 & 0:2  \\ \hline
13 & 6543245613452 & 3:35 & 2:14  \\ \hline
14 & 65432456134524 & 2:4 & 1:12  \\ \hline
15 & 654324561345243 & 1:2 &  0:1  \\ \hline
\multicolumn{4}{c}{}
\end{tabular}
\addcontentsline{toc}{section}{Table \ref{t:E6}: Singular loci of Schubert varieties in $E_6/P_6$}
\end{table}

\begin{table}[htb]
\caption{Schubert varieties and their singular loci in the Freudenthal variety.} \label{t:E7}
\begin{small}
\renewcommand{\arraystretch}{1.2}
\begin{tabular}{|c|l|l|l|}
\hline
  dim & $w$ & $\sfa:\ttJ$ & $\tSing_w$ \\ \hline\hline
     1 & 7 & 0:6 & \\ \hline
     2 & 76 & 0:5 & \\ \hline
     3 & 765 & 0:4 & \\ \hline
     4 & 7654 & 0:23 & \\ \hline
     5 & 76542 & 0:3 & \\ \hline
     5 & 76543 & 0:12 & \\ \hline
     6 & 765432 & 1:123 & 0:4  \\ \hline
     6 & 765431 & 0:2 & \\ \hline
     7 & 7654321 & 1:23 & 0:4  \\ \hline
     7 & 7654324 & 1:14 & 0:5  \\ \hline
     8 & 76543241 & 2:124 & 0:3  \\ \hline
     8 & 76543245 & 1:15 & 0:6  \\ \hline
     9 & 765432413 & 1:4 & 0:5  \\ \hline
     9 & 765432451 & 2:125 & 0:3  \\ \hline
     9 & 765432456 & 1:16 &  $o$  \\ \hline
     10 & 7654324513 & 3:145 & 1:23  \\ \hline
     10 & 7654324561 & 2:126 & 0:3  \\ \hline
     10 & 7654324567 & 0:1 & \\ \hline
     11 & 76543245134 & 2:35 & 0:2  \\ \hline
     11 & 76543245613 & 3:146 & 1:23  \\ \hline
     11 & 76543245671 & 1:12 & 0:3  \\ \hline
     12 & 765432451342 & 1:5 & 0:6  \\ \hline
\end{tabular}
\hspace{5pt}
\begin{tabular}{|c|l|l|l|}
\hline
  dim & $w$ & $\sfa:\ttJ$ & $\tSing_w$ \\ \hline\hline
     12 & 765432456134 & 4:1356 & 1:4  \\ \hline
     12 & 765432456713 & 2:14 & 1:23  \\ \hline
     13 & 7654324561342 & 3:156 & 1:4  \\ \hline
     13 & 7654324561345 & 2:36 & 0:2  \\ \hline
     13 & 7654324567134 & 3:135 & 1:4  \\ \hline
     14 & 76543245613452 & 4:356 & 3:146  \\ \hline
     14 & 76543245671342 & 2:15 & 1:4  \\ \hline
     14 & 76543245671345 & 3:136 & 2:35  \\ \hline
     15 & 765432456134524 & 3:46 & 2:126  \\ \hline
     15 & 765432456713452 & 5:1356 & 1:5, 2:14  \\ \hline
     15 & 765432456713456 & 1:3 & 0:2  \\ \hline
     16 & 7654324561345243 & 2:26 & 1:16  \\ \hline
     16 & 7654324567134524 & 4:146 & 1:5, 1:12  \\ \hline
     16 & 7654324567134562 & 3:35 & 2:14  \\ \hline
     17 & 76543245613452431 & 1:6 & $o$  \\ \hline
     17 & 76543245671345243 & 3:126 & 1:5, 0:1  \\ \hline
     17 & 76543245671345624 & 5:346 & 2:15  \\ \hline
     18 & 765432456713452431 & 2:16 & 1:5  \\ \hline
     18 & 765432456713456243 & 4:236 & 2:15  \\ \hline
     18 & 765432456713456245 & 2:4 & 1:12  \\ \hline
     19 & 7654324567134562431 & 3:36 & 2:15  \\ \hline
     19 & 7654324567134562453 & 5:246 & 3:35  \\ \hline
\end{tabular} \\
\vspace{10pt}
\begin{tabular}{|c|l|l|l|}
\hline
    dim & $w$ & $\sfa:\ttJ$ & $\tSing_w$ \\ \hline\hline
     20 & 76543245671345624531 & 4:46 & 3:35  \\ \hline
     20 & 76543245671345624534 & 3:25 & 1:3  \\ \hline
     21 & 765432456713456245341 & 5:256 & 2:4  \\ \hline
     21 & 765432456713456245342 & 1:2 & 0:1  \\ \hline
     22 & 7654324567134562453421 & 3:26 & 2:4  \\ \hline
     22 & 7654324567134562453413 & 2:5 & 1:3  \\ \hline
     23 & 76543245671345624534132 & 4:25 & 4:46  \\ \hline
     24 & 765432456713456245341324 & 3:4 & 3:36  \\ \hline
     25 & 7654324567134562453413245 & 2:3 & 2:16  \\ \hline
     26 & 76543245671345624534132456 & 1:1 & 1:6  \\ \hline
\end{tabular}  
\end{small}
\addcontentsline{toc}{section}{Table \ref{t:E7}: Singular loci of Schubert varieties in $E_7/P_7$}
\end{table}

\subsection{Codimension of the singular locus} 

Let $\e\in\Pi_{1,\sfa-1}$.  By Lemma \ref{L:we} and Theorem \ref{T:sing}, the irreducible component $X_{w_\e}\subset\tSing(X_w)$ has codimension $|\Delta(w,\e)|$.  In this section we characterize the Schubert varieties $X_w$ for which the codimension is minimal.  Zelevinski{\u\i} \cite{MR705051} showed that every Schubert variety in the Grassmannian admits a small resolution.  So, in particular, we know that $\tcodim_{X_w}X_{w_\e} \ge 3$ for all $X_w \subset \tGr(\tti,n+1)$.  This inequality also holds (and is sharp) for Schubert varieties in the spinor variety $\cS_n = D_n/P_n$.  On the other hand, the Lagrangian Grassmannian admits Schubert varieties with $\tcodim_{X_w}X_{w_\e} =2$.

\begin{corollary} \label{C:scd}
Let $X =G/P$ be cominuscule.  Fix $w \in W^\fp\backslash\{1,w_0\}$ with associated $\sfa,\ttJ$.  Let $\e\in\Pi_{1,\sfa-1}$.
\begin{a_list_emph}
\item  Suppose $X = \tGr(\tti,n+1) = A_n/P_\tti$.  Then $\tcodim_{X_w}X_{w_\e} \ge 3$.  Assume the notation of Section \ref{S:aJA}.   Equality holds if and only if there exist $0 < \ell \le \sfp$ and $0 < m \le \sfq$ such that $\ell+m = \sfa+1$ and $1=\ttj_\ell - \ttj_{\ell+1} = \ttk_{m+1} - \ttk_m$.
\item Suppose $X = \tLG(n,2n) = C_n/P_n$.  Then $\tcodim_{X_w}X_{w_\e} \ge 2$.  Assume the notation of Section \ref{S:aJC}.  Equality holds if and only if $\sfa= 2\ell-1 >0$ and $1 = \ttj_\ell - \ttj_{\ell+1}$.  In particular, these $X_w$ admit no small resolution.  
\item Suppose $X = \cS_n = D_n/P_n$.  Then $\tcodim_{X_w}X_{w_\e} \ge 3$.  Assume the notation of Section \ref{S:aJD}.  If $\ttj_1=n-1$ and $\sfa=1$, then equality holds if and only if $\ttj_2 = n-3$; if $\ttj_1 = n-1$ and $\sfa=2$, then equality holds if and only if $\ttj_2 = n-2$ and $\ttj_3 = n-4$.  In all other cases, equality holds if and only if there exist $\a_{n-1}(Z_w) < \ell < m$ such that $\ell+m = \sfa+1+\a_{n-1}(Z_w)$, and $1 + \d_{\ell\sfr} = \ttj_\ell - \ttj_{\ell+1}$ and $1 = \ttj_m - \ttj_{m+1}$.
\end{a_list_emph}
\end{corollary}


\begin{proof}[Proof of Corollary \ref{C:scd}(a)]
The elements of $\Pi_{1,\sfa-1}$ are of the form 
$$
  \e \ = \ \a_{\ttj_\ell+1} + \cdots + \a_{\ttk_m-1} \,, \quad\hbox{with}\quad
  \ell+m = \sfa+1 \quad\hbox{and}\quad 0 \le \ell,m\, .
$$
Both $\n_1 = \a_{\ttj_\ell}$ and $\n_2 =\a_{\ttk_m}$ are elements of $\Delta(\fg_{0,1})$, and $\e$, $\e+\n_1$ and $\e+\n_2$ are distinct elements of $\Delta(w,\e)$.  Thus $\tcodim_{X_w}X_{w_\e} = |\Delta(w,\e)| \ge 3$.  Additionally, $|\Delta(w,\e)| =3$ if and only if $1 = \ttj_\ell - \ttj_{\ell+1} = \ttk_{m+1} - \ttk_m$.
\end{proof}

\begin{proof}[Proof of Corollary \ref{C:scd}(b)]
The elements of $\Pi_{1,\sfa-1}$ are of the form $\e' = \a_{\ttj_\sfa+1} + \cdots + \a_n$, or
$$
  \e \ = \ \a_{\ttj_\ell+1} + \cdots + \a_{\ttj_m} + 
  2 ( \a_{\ttj_m+1} + \cdots + \a_{n-1}) + \a_n \,, 
$$
with $\ell+m = \sfa+1$ and $0 < m \le \ell$.  For $\e'$, observe that $\n_1=\a_{\ttj_\sfa}\in\Delta(\fg_{0,1})$ and $\e'$, $\e' + \n_1 \in \Delta(w,\e')$.  So $\tcodim_{X_w}X_{w_\e'} = |\Delta(w,\e')| \ge 2$.  Equality holds if and only if $1=\ttj_\sfa-\ttj_{\sfa+1}$ and $\sfa=1$.  (If $\sfa>1$, then $\n_2 = \a_{\ttj_1} + \cdots + \a_{n-1} \in \Delta(\fg_{0,1})$ and $\e' + \n_2 \in \Delta(w,\e)$ is distinct from $\e'$ and $\e'+\n_1$.)

For $\e$ with $\ell < m$, observe that both $\n_1 = \a_{\ttj_\ell}$ and $\n_2 =\a_{\ttj_m}$ are elements of $\Delta(\fg_{0,1})$, and $\e$, $\e+\n_1$ and $\e+\n_2$ are distinct elements of $\Delta(w,\e)$.  Thus $\tcodim_{X_w}X_{w_\e} = |\Delta(w,\e)| \ge 3$.  If $\ell=m$, then $\n_1=\n_2$.  So, $\tcodim_{X_w}X_{w_\e} = |\Delta(w,\e)| \ge 2$.  Equality holds if and only if $1 = \ttj_\ell - \ttj_{\ell+1}$.
\end{proof}

\noindent The proof of Corollary \ref{C:scd}(c), which is very similar to, though more involved than, that of Corollary \ref{C:scd}(b), is left to the reader.




\section{Proof of Theorem \ref{T:sing}} \label{S:prf}

\subsection{The stabilizer of \boldmath $Y_w$ \unboldmath } \label{S:stab}

Virtue of Proposition \ref{P:BP}, the stabilizer $\tStab(Y_w) = \{ g\in G \ | \ g Y_w = Y_w \}$ will play an important r\^ole in the proof of Theorem \ref{T:sing}.

\begin{proposition}[Brion--Polo {\cite{MR1703350}}] \label{P:BP}
The smooth locus $Y_w^0$ of $Y_w$ is the orbit of $o \in X$ under the stabilizer $G_w$.
\end{proposition}

\noindent In this section we will apply Proposition \ref{P:aJ} to obtain a description of the stabilizer in terms of the the data $(\sfa(w),\ttJ(w))$.  Review the definitions of the grading elements $Z_\tti$ and $Z_w$, and the $(Z_\tti,Z_w)$--bigraded decomposition $\fg = \op\,\fg_{j,k}$, in Sections \ref{S:not} and \ref{S:aJ}.  

\begin{lemma} \label{L:gw}
The largest subalgebra $\fg_w \subset \fg$ containing $\fn_w$ and such that $\fg_w \equiv \fn_w$ mod $\fg_{\ge0}$ is 
\begin{equation} \label{E:gw}
  \fg_w \ = \ \fn_w \,\op\, \fg_{0,\ge0} \,\op\, \fg_{1 , \ge \sfa} \,.
\end{equation}
\end{lemma}

\noindent It is well-known that the stabilizer of the Schubert variety $Y_w$ is parabolic; see, for example, \cite{MR1782635} or \cite{MR1703350}.   The following is a corollary of Lemma \ref{L:gw}.

\begin{corollary} \label{C:stab}
The stabilizer of $Y_w$ in $G$ is the parabolic subgroup $G_w$ associated with $\fg_w$.
\end{corollary}


\begin{proof}[Proof of Lemma \ref{L:gw}]
Let $\fg'$ denote the right-hand side of \eqref{E:gw}.  It is clear that $\fg'$ is a subalgebra of $\fg$, and that $\fg' \equiv \fn_w$ mod $\fg_{\ge0}$.  So $\fg'\subset\fg_w$.  By construction $\fg_{0,\ge0}$ is the stabilizer of $\fn_w$ in $\fg_0$, see Section \ref{S:aJ}.  Consequently, $\fg_w \equiv \fn_w \op \fg_{0,\ge0}$ mod $\fg_1$.  So it remains to see that,
\begin{equation} \label{E:gwprf}
  \fg_{1,<\sfa}\,\cap\, \fg_w \ = \ 0 \,.
\end{equation}

Assume the converse: suppose there exists a nonzero a nonzero $\z\in\fg_{1,<\sfa}\,\cap\, \fg_w$.  Since $\fg_1$ is an irreducible $\fg_0$--module, and $\fg_{0,\ge0} \op\fg_{1,\ge\sfa}\subset \fg_w$, we may assume without loss of generality that $\z\in\fg_{1,\sfa-1}$.  Then $[\z,\fn_w] \subset \fg_w$ holds if and only if $[\z,\fg_{-1,-\sfa}] = 0$.  In particular, $U = \{ \z \in \fg_{1,\sfa-1} \ | \ [\z,\fg_{-1,-\sfa}] = 0 \}\not=0$.  

The Jacobi identity implies $U$ is a $\fg_{0,0}$--module.  So there exists $\Delta(U) \subset \Delta(\fg_{1,\sfa-1})$ such that $U = \oplus_{\a\in\Delta(U)} \fg_\a$.  Let $\c \in \Delta(U)$ be a highest $\fg_{0,0}$--weight.  Let $\tilde \a$ be the highest root of $\fg$.  Then there exists a sequence $\s_1,\ldots,\s_\ell \in \Sigma$ of simple roots such that $\c_i := \tilde \c - \s_1 - \cdots - \s_i$ is a root, for each $1 \le i \le \ell$, and $\c_\ell = \c$.  Since $\c$ is a highest $\fg_{0,0}$--weight, and $\c_\ell + \s_\ell = \c_{\ell-1}$ is a root, $\s_\ell$ must lie in $\Delta(\fg_{0,1})$.  In particular, $\c_{\ell-1} \in \Delta(\fg_{1,\sfa})$.  From $\c - \c_{\ell-1} = -\s_\ell$, we see that $0 \not= \fg_{-\s_\ell} = [\fg_\c , \fg_{-\c_{\ell-1}}] \subset [\fg_\c , \fg_{-1,-\sfa}]$, implying $\fg_\c \not\subset U$, a contradiction.  We conclude that \eqref{E:gwprf} must hold.
\end{proof}

The algebra $\fg_w$ is a nonstandard parabolic; that is, $\fg_w$ does not contain the fixed Borel subalgebra $\fb\subset\fg$.  
Nonetheless, it admits a useful description in terms of the element 
\begin{subequations} \label{SE:tilde}
\begin{equation} \label{E:tZw}
  \tilde Z_w \ = \ Z_w - \sfa Z_\tti \ \in \ \fh \,.
\end{equation}
As an element of the Cartan subalgebra, $\tilde Z_w$ acts on $\fg$ by eigenvalues.  Set $\sft = \tmax\{ \a(\tilde Z_w) \ | \ \a\in\Delta\}$.  The \emph{$\tilde Z_w$--graded decomposition} of $\fg$ is the eigenspace decomposition
\begin{equation} \label{E:tg}
  \fg \ = \ \bigoplus_{s=-\sft}^\sft \, \tilde \fg_s \, \quad \hbox{where}\quad
  \tfg_s \, := \, \{ \z \in \fg \, | \, [\tilde Z_w \,,\, \z ] = s \z \} \,.
\end{equation}
Given \eqref{E:gw}, it is straight-forward to confirm that 
\begin{equation} \label{E:gwP}
  \fg_w = \tilde\fg_{\ge0} \, .
\end{equation} 
\end{subequations}
(Note that this provides a second proof that $\fg_w$ is parabolic, cf. \cite[Theorem 3.2.1(2)]{MR2532439}.)  

\begin{remark} \label{R:Gw}
We see from \eqref{SE:tilde} that the pair $(\sfa,\ttJ)$ encodes the relationship between the two parabolic subalgebras $\fp$ and $\fg_w$.  Moreover, \eqref{E:gkl} and \eqref{SE:tilde} yield $\fg_{0,0} = \fg_0 \,\cap\,\tilde\fg_0$ and $\fg_{1,\sfa-1} = \fg_1 \,\cap\,\tilde\fg_{-1}$.  So Corollary \ref{C:|sing|} may be interpreted as saying the following: to the pair of parabolic subalgebras $\fp = \fg_{\ge0}$ and $\fg_w = \tilde\fg_{\ge0}$ is naturally associated a reductive subalgebra $\fr = \fg_0 \cap \tilde\fg_0$ and a $\fr$--module $U = \fg_1 \cap \tilde\fg_{-1}$, such that the irreducible components of the singular locus of $Y_w$ are in bijection with the irreducible $\fr$--submodules of $U$.
\end{remark}


\subsection{Lemmas} \label{S:lemmas}
We begin with a proof of Lemma \ref{L:we}.  The remainder of the section is then devoted to the proof of Theorem \ref{T:sing}; the theorem is an immediate corollary of Lemmas \ref{L:wemax} and \ref{L:singXw}.

\begin{remark} \label{R:Dw}
One important consequence of Remark \ref{R:abelian} is that given a set $\Phi \subset \Delta(\fg_1)$, there exists $w \in W^\fp$ such that $\Delta(w) = \Phi$ if and only if $\Delta^+ \backslash \Phi$ is closed.  For details see \cite[Section 2.3]{MR2960030}.
\end{remark}

\begin{proof}[Proof of Lemma \ref{L:we}]
By Remark \ref{R:abelian} it suffices to show that  
\begin{eqnarray*}
  \Phi(w,\e) \ := \ 
  \Delta^+ \backslash \{ \Delta(w) \backslash \Delta(w,\e) \} & = & 
  \{ \Delta^+ \backslash \Delta(w) \} \ \sqcup \ \Delta(w,\e) \\
  & = & \Delta(\fg_{1,>\sfa}) \ \sqcup \ \Delta^+(\fg_{0,\ge0}) \ 
  \sqcup \ \Delta(w,\e) 
\end{eqnarray*}
is closed.  By Remark \ref{R:abelian}, the set $\Delta(\fg_{1,>\sfa}) \sqcup \Delta(w,\e) \subset \Delta(\fg_1)$ is closed.  Similarly, by Remark \ref{R:Dw}, the set $\Delta^+\backslash\Delta(w) = \Delta(\fg_{1,>\sfa}) \sqcup \Delta^+(\fg_{0,\ge0})$ is closed.  So it remains to show that given roots $\n\in\Delta(w,\e)$ and $\b \in \Delta^+(\fg_{0,\ge0})$ such that $\n+\b$ is also a root, it is the case that $\n+\b \in \Phi(w,\e)$.  There are two cases to consider: either $\n = \e$, or $\n = \e + \n'$ for some $\n' \in \Delta(\fg_{0,1})$. \smallskip

\noindent {\small {\bf (I)}} Assume $\n=\e$ and $\n+\b\in\Delta$.  
\begin{circlist}
\item
If $\b \in \Delta(\fg_{0,>1})$, then $\n+\b \in \Delta(\fg_{1,>\sfa}) \subset \Phi(w,\e)$.  
\item
If $\b \in \Delta(\fg_{0,1})$, then $\n+\b\in\Delta(w,\e) \subset \Phi(w,\e)$.  \item
If $\b \in \Delta^+(\fg_{0,0})$, then $\n+\b = \e+\b$ cannot be a root because $\e$ is a highest $\fg_{0,0}$--weight.
\end{circlist}

\noindent {\small {\bf (II)}} Assume $\n=\e+\n' \in\Delta(\fg_{1,\sfa})$ and $\n+\b\in\Delta$.  
\begin{circlist}
\item
If $\b \in \Delta(\fg_{0,>0})$, then $\n+\b \in \Delta(\fg_{1,>\sfa}) \subset \Phi(w,\e)$.  
\item 
If $\b \in \Delta^+(\fg_{0,0})$, then $\n+\b = \e+\n'+\b \in \Delta$; therefore,
\begin{eqnarray*}
  \{0\} \ \not= \ \fg_{\n+\b} & = & \left[ \fg_\n \,,\, \fg_\b \right] 
  \ = \ \left[ [\fg_\e \,,\, \fg_{\n'} ] \,,\, \fg_\b \right] \\
  & = & \left[ [\fg_\b \,,\, \fg_{\n'} ] \,,\, \fg_\e \right] \ + \ 
  \left[ [\fg_\e \,,\, \fg_\b ] \,,\, \fg_{\n'} \right] \, .
\end{eqnarray*} 
Because $\e$ is a highest $\fg_{0,0}$--weight, the bracket $[\fg_\e \,,\, \fg_\b ]$ is zero.  This forces $[\fg_\b \,,\, \fg_{\n'} ]$ to be nonzero.  Equivalently, $\b+\n' \in \Delta(\fg_{0,1})$.  Thus $\n+\b = \e + (\n'+\b) \in \Delta(w,\e)$.
\end{circlist}
\end{proof}

\begin{lemma} \label{L:wemax}
The set $\{ \Delta(w_\e) \ | \ \e \in \Pi_{1,\sfa-1} \}$ is precisely the collection of $\Delta(w_1) \subset \Delta(w)$, with $w_1 \in W^\fp$, that are maximal with the property that 
\begin{equation} \label{E:wemax}
\Delta(w) \backslash \Delta(w_1) \not\subset \Delta(\fg_{1,\sfa}) \, .
\end{equation}
\end{lemma}

\begin{lemma}\label{L:singXw}
Given $w_1 < w$, we have $X_{w_1} \subset \tSing(X_w)$ if and only if $\Delta(w) \backslash \Delta(w_1) \not\subset \Delta(\fg_{1,\sfa})$.
\end{lemma}

Given a root $\c$, let $r_\c \in W$ denote the associated reflection.  In order to prove Lemmas \ref{L:wemax} and \ref{L:singXw} we first recall

\begin{lemma} \label{L:w'w}
Let $w_1 , w \in W^\fp$ be elements of the Hasse diagram of a cominuscule $G/P$.  Then $w_1 \le w$ if and only if $\Delta(w_1) \subset \Delta(w)$.  In this case there is an ordering $\{ \c_1 \, , \, \c_2 \, , \ldots , \, \c_m \}$ of the elements of $\Delta(w) \backslash \Delta(w_1)$ so that $w_{\ell+1} = r_{\c_\ell} w_{\ell} \in W^\fp$, $w_{m+1} = w$ and $\Delta(w_{\ell+1}) = \Delta(w_1) \sqcup\{\c_1,\ldots,\c_\ell\}$.
\end{lemma}

\begin{proof}
This well-known result may be deduced from Propositions 3.2.12(5) and 3.2.15(3) of \cite{MR2532439}.  Note that their $\Phi_w$ is our $\Delta(w)$.
\end{proof}

\begin{corollary} \label{C:we}
The Weyl group element $w_1 = r_\e w_\e$ is an element of $W^\fp$ and $\Delta(r_\e w_\e) = \Delta(w_\e) \sqcup\{\e\}$.  Moreover, there exists an ordering $\{\n_1,\cdots,\n_m\}$ of the elements of $\Delta(w,\e)\backslash\{\e\}$ so that $w_{\ell+1} := r_{\n_\ell}\cdots r_{\n_1}w_1 \in W^\fp$ and $\Delta(w_{\ell+1}) = \Delta(w_\e) \sqcup \{ \e , \n_1 ,\ldots, \n_\ell \}$, for all $1 \le \ell \le m$.
\end{corollary}

\begin{proof}
It suffices to observe that, in the ordering of the roots $\Delta(w,\e) = \Delta(w)\backslash\Delta(w_\e)$ given by Lemma \ref{L:w'w}, it is necessarily the case that $\c_1 = \e$.  By Remark \ref{R:Dw}, the set $\Phi = \Delta^+\backslash\Delta( r_{\c_1}w_\e)$ is closed.  Suppose that $\c_1 \not= \e$.  Then, by the definition \eqref{E:we} of $\Delta(w,\e)$, there exists $\m \in \Delta(\fg_{0,1})$ such that $\c_1 = \e+\m$.  However, $\e,\m\in\Phi$, while $\c_1 \not\in\Phi$, contradicting the closure of $\Phi$. 
\end{proof}



\begin{proof}[Proof of Lemma \ref{L:wemax}]
Recall (Proposition \ref{P:aJ}) that $\Delta(w) = \Delta(\fg_{1,\le\sfa})$.  By Lemma \ref{L:we}, $\Delta(w) \backslash\Delta(w_\e) = \Delta(w,\e)$.  The definition \eqref{E:we} yields $\Delta(w,\e) \cap \Delta(\fg_{1,<\sfa}) = \{\e\}$.  So $\Delta(w,\e) \not\subset \Delta(\fg_{1,\sfa})$.  To see that $\Delta(w_\e)$ is maximal with respect to \eqref{E:wemax}, recall (Remark \ref{R:Dw}) that $\Delta^+ \backslash\Delta(w_\e)$ is closed; this forces $\Delta(w,\e) \subset \Delta^+\backslash\Delta(w_\e)$.

Conversely, suppose that $\Delta(w_1) \subset \Delta(w)$ satisfies \eqref{E:wemax}.  Fix $\m = \m_0 \in \Delta(\fg_{1,<\sfa}) \backslash\Delta(w_1)$.  There exists a sequence of simple roots $\s_1,\ldots,\s_\ell \in \Sigma$ such that each $\m_i := \m + \s_1 + \cdots + \s_i$ is a root, for all $1\le i\le \ell$, and $\m_\ell$ is the highest root of $\fg$.  Since both $\m$ and $\m_\ell$ lie in $\Delta(\fg_1)$, the simple roots $\s_i$ must lie in $\Delta^+(\fg_0)$.  It follows from Remark \ref{R:Dw} that each $\m_i \in \Delta^+\backslash\Delta(w_1)$.  Moreover, since $\m_\ell$ is the highest root of $\fg$, at least one of the $\m_i$ is an element of $\Delta(\fg_{1,\sfa-1})$.  Let $U \subset \fg_{1,\sfa-1}$ be the irreducible $\fg_{0,0}$--submodule containing $\fg_{\m_i}$.  Let $\e \in \Delta(U)$ be the highest $\fg_{0,0}$--weight of $U$.  There exists a second sequence of simple roots $\s'_1,\ldots,\s'_m \in \Sigma(\fg_{0,0})$ such that each $\m_{i,k} := \m_i + \s'_1 + \cdots + \s'_k \in \Delta(U)$, with $1 \le k \le m$, and $\m_{i,m} = \e$.  Since $\m_i \in \Delta^+\backslash\Delta(w_1)$, Remark \ref{R:Dw} implies $\e \in \Delta^+\backslash\Delta(w_1)$.  As in the first paragraph of this proof, Remark \ref{R:Dw} forces $\Delta(w,\e) \subset \Delta^+\backslash\Delta(w_1)$.  Thus, $\Delta(w_1) \subset \Delta(w_\e)$.
\end{proof}

\begin{proof}[Proof of Lemma \ref{L:singXw}]
First we will show that the lemma is equivalent to \eqref{E:lem3}.  Recall from \eqref{E:Xw} that $Y_w = wX_w$.  So the lemma is equivalent to
\begin{subequations}
\begin{equation}\label{E:lem1}
  w X_{w_1} \subset \tSing(Y_w) \quad
  \hbox{if and only if}\quad
   \Delta(w) \backslash \Delta(w_1) \not\subset \Delta(\fg_{1,\sfa}) \, .
\end{equation}
By Lemma \ref{L:w'w} we have $w = \tau\,  w_1$, where $\tau = r_{\c_m}\cdots r_{\c_2}r_{\c_1}$ and $\{ \c_1,\c_2,\ldots,\c_m\} = \Delta(w) \backslash\Delta(w_1)$.  So $w X_{w_1} = \tau Y_{w_1}$, and \eqref{E:lem1} is equivalent to 
\begin{equation}\label{E:lem2}
  \tau Y_{w_1} \ \subset \ \tSing(Y_w) \quad\hbox{if and only if}\quad
  \Delta(w) \backslash \Delta(w_1) \not\subset \Delta(\fg_{1,\sfa}) \,.
\end{equation}
By Lemma \ref{L:w'w} we have $\Delta(w_1) \subset \Delta(w)$.  Equations \eqref{E:nw} and \eqref{E:gw} then imply $\fn_{w_1} \subset \fn_w\subset\fg_w$, and therefore $N_{w_1} \subset N_w \subset G_w$.  By Proposition \ref{P:BP}, $G_w \cdot o = Y_w \backslash \tSing(Y_w) = Y_w^0$.  So $N_{w_1}\cdot o \subset G_w \cdot o = Y_w^0$.  Since $Y_{w_1} = \overline{N_{w_1}\cdot o}$, we see that $\tau Y_{w_1} \subset \tSing(Y_w)$ if and only if $\tau N_{w_1} \cdot o \,\not\subset\, G_w \cdot o$.  Therefore, \eqref{E:lem2} is equivalent to 
\begin{equation} \label{E:lem3}
  \tau N_{w_1} \cdot o \ \not\subset \ G_w \cdot o
  \quad\hbox{if and only if}\quad
  \Delta(w) \backslash \Delta(w_1) \not\subset \Delta(\fg_{1,\sfa})\,.
\end{equation}
\end{subequations}

Let $\c \in \Delta(w)\backslash\Delta(w_1)$.  As an element of $W = N_G(H)/H$, the reflection $r_\c$ is represented by $\texp(\xi)\texp(\z)\texp(\xi) \in N_G(H)$, where the $\x\in\fg_{\c}$ and $\z\in\fg_{-\c}$ are scaled so that $\c([\x,\z]) = -2$; see, for example, the proof of \cite[Theorem 3.2.19(1)]{MR2532439}.  Lemma \ref{L:gw} and Corollary \ref{C:stab} imply that 
\begin{equation} \label{E:r}
  r_\c \in G_w
  \quad\hbox{if and only if}\quad
  \c \in \Delta(\fg_{1,\sfa}) \,.
\end{equation}  
In this case, $r_\c N_{w_1} \subset r_\c G_{w} = G_w$.  This establishes one direction of \eqref{E:lem3}: if $\Delta(w) \backslash \Delta(w_1) \subset \Delta(\fg_{1,\sfa})$, then $\tau \in G_w$ and $\tau N_{w_1} \cdot o \subset  G_w \cdot o = Y_w \backslash \tSing(Y_w)$.

Suppose $\Delta(w)\backslash\Delta(w_1) \not \subset \Delta(\fg_{1,\sfa})$.  By Lemma \ref{L:wemax} there exists $\e\in\Pi_{1,\sfa-1}$ such that $\Delta(w_1) \subset \Delta(w_\e)$.

\begin{claim*}
If $r_\e \not\in G_w P$, then $\tau N_{w_1} \cdot o \not\subset G_w \cdot o$.
\end{claim*}

\noindent Assume that claim holds.  Then to establish the second direction of \eqref{E:lem3}, it remains to show that the reflection $r_\e \in W$ can not be represented by an element $\tilde p p \in G_w P$ with $\tilde p \in G_w$ and $p \in P$. 

Recall the $Z_\tti$--graded decomposition \eqref{E:gi} of $\fg$.  Let $G_0 := \{ g \in G \ | \ \tAd_g(\fg_j) \subset \fg_j\}$.  Then $G_0$ is a closed subgroup of $G$ with Lie algebra $\fg_0$.  By \cite[Theorem 3.1.3]{MR2532439}, the map $\fg_1 \times G_0 \to P$ sending $(u,g) \mapsto \texp(u)g$ is a diffeomorphism.  Likewise, recall the $\wt Z_w$--graded decomposition \eqref{E:tg} of $\fg$. Again, $\wt G_0 := \{ g \in G \ | \ \tAd_g(\tilde\fg_k) \subset \tilde\fg_k\}$ is a closed subgroup of $G$ with Lie algebra $\tilde \fg_0$.   By \eqref{E:gwP}, $\fg_w = \tilde \fg_{\ge0} = \tilde\fg_0 \op \tilde\fg_+$, and \cite[Theorem 3.1.3]{MR2532439} implies that the map $\tilde \fg_+ \times \wt G_0  \to G_w$ sending $(\tilde u,\tilde g) \mapsto \texp(\tilde u)\tilde g$ is a diffeomorphism.  

We will argue by contradiction, supposing that $\tilde p p \in G_w P$ represents the reflection $r_\e$.  In particular, $\tAd_{\tilde p p} : \fg \to \fg$ preserves the Cartan subalgebra $\fh$.  Write $\tilde p = \texp(\tilde u) \tilde g$, with $\tilde u \in \tilde \fg_+$ and $\tilde g \in \wt G_0$, and $p = \texp(u)g$, with $g \in G_0$ and $u \in \fg_1$. Fix $H \in \fh$, and define $h_s \in \tilde \fg_s$ by $\tAd_p H = \sum_{s=-\sft}^\sft \tilde h_s$, and set $\tilde h_s = \tAd_{\tilde g} h_s \in \tilde \fg_s$.  Define $\tilde h_{s,r} \in \tilde \fg_r$ by $\tAd_{\texp(\tilde u)} \tilde h_s = \tilde h_{s,s} + \tilde h_{s,s+1} + \cdots + \tilde h_{s,\sft}$, and note that $\tilde h_{s,s} = \tilde h_s$.  Then 
$$ \textstyle
  \tAd_{\tilde p p} H \ = \ \sum_{s=-\sft}^\sft \sum_{r=s}^\sft \tilde h_{s,r}
  \ = \ \sum_{r=-\sft}^\sft \tilde H_r \,, 
$$
where $\tilde H_r := \sum_{s=-\sft}^r \tilde h_{s,r} \in \tilde \fg_r$.  

Since $\tAd_{\tilde p p}$ preserves $\fh$, and $\fh\subset\tilde\fg_0$, it must be the case that 
\begin{equation}\label{E:lem4}
  \tilde H_0 \in \fh \,, \quad \hbox{ and }\quad \tilde H_r = 0 \,, \ \hbox{ when $r \not=0$.}
\end{equation}  
In particular, $\tilde H_{-\sft} = \tilde h_{-\sft,-\sft} = \tilde h_{-\sft} = 0$.  This in turn yields $\tilde h_{-\sft,r}=0$ for all $r$.  Moreover, since $\tilde h_{-\sft} = \tAd_{\tilde g} h_{-\sft}$, we also have $h_{-\sft}=0$.  Next, $0=\tilde H_{1-\sft} = \tilde h_{-\sft,1-\sft} + \tilde h_{1-\sft,1-\sft} = \tilde h_{1-\sft}$.  As above, this implies $\tilde h_{1-\sft,r}=0$, for all $r$, and $h_{1-\sft}=0$.  Continuing by induction, we see that 
\begin{equation}\label{E:lem5}
  h_s=0 \quad \hbox{for all} \quad s < 0 \,.
\end{equation}
In particular,  $\tAd_p H \in \tilde \fg_{\ge0}$.  Our choice of $H \in \fh$ was arbitrary, so $\tAd_p \fh \subset \tilde \fg_{\ge0} = \fg_w$.  This implies $p \in G_w$.  In particular, $\tilde p p \in G_w$.  

This yields a contradiction as follows.  Note that $\fg_{-\e} \subset \tilde \fg_{1}$.  So given any $q \in G_w$, we have $\tAd_q \fg_{-\e} \subset \tilde \fg_{\ge1}$. On the other hand, $\fg_\e \subset \tilde \fg_{-1}$, and $\tAd_{r_\e}(\fg_{-\e}) = \fg_\e$.  Therefore, there exists no element $q \in G_w$ such that $\tAd_q (\fg_{-\e}) = \tilde \fg_\e$.  Modulo the claim, this completes the proof of Lemma \ref{L:singXw}.






\medskip

\noindent\emph{Proof of claim.}  By Lemma \ref{L:w'w}, $w_1 \le w_\e$.  Therefore, $X_{w_1} \subset X_{w_\e}$.  So to see that $X_{w_1} \subset \tSing(X_w)$, it suffices to show that $X_{w_\e} \subset \tSing(X_w)$.  Equivalently, as discussed above, $\tau N_{w_\e} \cdot o \,\not\subset\, G_w\cdot o$, where $\tau = r_{\n_m}\cdots r_{\n_1} r_\e$ is as given by Corollary \ref{C:we}.  Since $\n_\ell \in \Delta(\fg_{1,\sfa})$, we have $r_{\n_\ell} \in G_w$, by \eqref{E:r}.  So $\tau N_{w_\e} \cdot o \not\subset G_w \cdot o$ if and only if $r_\e N_{w_\e} \cdot o \not\subset G_w \cdot o$.  In particular, to see that $\tau N_{w_\e}\cdot o \not\subset G_w \cdot o$, it suffices to show that $r_\e \cdot o \not\in G_w \cdot o$.  Lifting to $G$, the latter is equivalent to $r_\e \not \in G_w P$.  
\end{proof}

\appendix\section{Geometric descriptions of $Y_{\sfa,\ttJ}$} \label{S:app1}

The classical cominuscule $G/P$ admit geometric, partition based descriptions.  A `dictionary' relating these descriptions to the representation theoretic $(\sfa,\ttJ)$--description (Section \ref{S:aJ}) is given in \cite[Appendix A]{flex}.  We now briefly summarize those results for the reader's convenience.

\subsection{Notation} \label{A:not}
Given a vector space $V \simeq\bC^m$, we fix a basis $\{ e_1 , \ldots , e_m\}$.  Let $\{ e^1 , \ldots , e^m\}$ denote the dual basis of $V^*$.  Set 
$$
  e^k_\ell \ \dfn \  e_\ell \ot e^k \ \in \ \tEnd(V)
  \quad\hbox{for all} \quad \le k,\ell \le m \,.
$$

$\bullet$ When $V \simeq \bC^{2n+1}$ is of odd dimension and admits a nondegenerate symmetric bilinear form $(\cdot,\cdot)$, then we will normalize the basis so that $(e_k , e_\ell) = (e_{n+k} , e_{n+\ell}) = (e_k , e_{2n+1} ) = (e_{n+k} , e_{2n+1}) = 0$, $(e_k , e_{n+\ell} ) = \d_{k\ell}$, for all $1 \le k,\ell\le n$, and $(e_{2n+1} , e_{2n+1})=1$.  

$\bullet$ When $V\simeq\bC^{2n}$ is of even dimension and admits a nondegenerate (symmetric or skew-symmetric) bilinear form $(\cdot,\cdot)$, we normalize the basis so that  $(e_k , e_\ell) = (e_{n+k} , e_{n+\ell}) = 0$ and $(e_k , e_{n+\ell} ) = \d_{k\ell}$, for all $1 \le k,\ell\le n$.  In this setting, we fix an isotropic flag $F^\sbullet$ in $\bC^{2n}$ by specifying $F^k = \langle e_1,\ldots,e_k\rangle$ and $(F^k,F^{2n-k}) = 0$ for $1\le k \le n$.

\subsection{Odd dimensional quadrics \boldmath $Q^{2n-1} = B_n/P_1$ \unboldmath } \label{S:aJodd}
Set $m = 2n-1$.  There is a bijection between $W^\fp\backslash\{1,w_0\}$ and pairs $\sfa,\ttJ$ such that $\ttJ = \{ \ttj \} \subset \{2,\ldots,n\}$ and $\sfa \in \{0,1\}$; see \cite[Corollary 3.17]{MR2960030}.  If $\sfa = 0$, then 
$$
  Y_w \ = \ \bP \langle e_1 , \ldots,e_\ttj \rangle = \bP^{\ttj-1}\,.  
$$
If $\sfa=1$, then 
$$
  Y_w \ = \ Q^m \cap \bP\langle e_1 , \ldots , e_{n+1} , e_{n+\ttj+1} , \ldots , e_{2n+1} \rangle 
  \,.
$$

\subsection{Even dimensional quadrics \boldmath $Q^{2n-2} = D_n/P_1$\unboldmath} \label{S:aJeven}
Set $m = 2n-2$.  There is a bijection between $W^\fp\backslash\{1,w_0\}$ and pairs $\sfa,\ttJ$ such that either
\begin{circlist}
\item  $\sfa=0$ and $\ttJ = \{ \ttj\} \subset \{2,\ldots,n\}$ or $\ttJ = \{n-1,n\}$; or
\item $\sfa=1$ and $\ttJ = \{\ttj\} \subset \{2,\ldots,n-2\}$ or $\ttJ = \{n-1,n\}$.
\end{circlist} 
See \cite[Corollary 3.17]{MR2960030}.  First suppose that $\sfa=0$.  If $\ttJ = \{ \ttj\}$ with $2 \le \ttj \le n-2$, then $X_w = \bP^{\ttj-1}$.  If $\ttJ=\{n-1\}$ or $\ttJ=\{n\}$, then $X_w = \bP^{n-1}$.  If $\ttJ = \{ n-1 , n \}$, then $X_w = \bP^{n-2}$.

Next suppose that $\sfa=1$.  If $\ttJ = \{ \ttj \}$ with $2 \le \ttj \le n-2$, then $X_w = Q^m \cap \bP \langle e_1 , \ldots , e_{n+1} , e_{n+\ttj+1} , \ldots , e_{2n}  \rangle$.  If $\ttJ = \{n-1,n\}$, then $X_w = Q^m \cap \bP\langle e_1 , \ldots , e_{n+1} , e_{2n} \rangle$.

\subsection{Grassmannians $\tGr(\tti,n+1) = A_n/P_\tti$} \label{S:aJA}
There is a bijection between $W^\fp\backslash\{1,w_0\}$ and pairs $\sfa,\ttJ$ such that $\ttJ = \{ \ttj_\sfp \,,\,\ldots \,,\, \ttj_1 \,,\, \ttk_1 \,,\, \ldots \,,\, \ttk_\sfq \} \subset \{1,\ldots,n\} \backslash\{\tti\}$ is ordered so that 
$$
  1 \ \le \ \ttj_\sfp \,<\,\cdots \,<\, \ttj_1 \,< \ \tti \ < 
  \, \ttk_1 \,<\, \cdots \,<\, \ttk_\sfq \ \le \ n \,,
$$
and satisfying $\sfp,\sfq \in \{\sfa, \sfa+1\}$; see \cite[Corollary 3.17]{MR2960030}.  (Beware, these $\sfp,\sfq$ do not agree with those of \cite{MR2960030}, cf. Remark \ref{R:newaJ}.)  For convenience we set
$$
  \ttj_{\sfp+1} \, := \, 0 \, , \quad \ttj_0 \, := \, \tti \, =: \, \ttk_0 \,, 
  \quad \ttk_{\sfq+1} \, := \, n+1 \, .
$$

It is well-known that Schubert varieties in $X=\tGr(\tti, n+1)$ are indexed by partitions 
\begin{equation} \label{E:Apart}
  \lambda \,=\, (\lambda_1,\ldots,\lambda_\tti) \in \bZ^\tti 
  \quad\hbox{such that}\quad
  1 \le \lambda_1 < \lambda_2 < \cdots < \lambda_\tti \le n+1 \,,
\end{equation}
cf. \cite[\S3.1.3]{MR1782635}.  Fix a flag $0 \subset F^1 \subset F^2 \subset\cdots\subset F^{n+1}$.  The corresponding Schubert variety is
\begin{equation}\label{E:Y1}
  Y_\lambda(F^\sbullet) \ := \ \{ E \in X \ | \ 
  \tdim( E \cap F^{\lambda_k}) \ge k \,,\ \forall \ k \} \,.
\end{equation}
Note that, if $\lambda_{k+1} = \lambda_k+1$, then the condition $\tdim(E\cap F^{\lambda_k}) \ge k$ is redundant; it is implied by $\tdim(E\cap F^{\lambda_{k+1}}) \ge k+1$.  To remove the redundancies, decompose $\lambda = \m_\sfp\cdots\m_1\m_0$ into maximal blocks of consecutive integers.  For example, if $\lambda = (2,3,4,7,8,12)$, then $\m_2 = (2,3,4)$, $\m_1 = (7,8)$ and $\m_0 = (12)$.  Let 
\begin{equation} \label{E:jA} 
  \ttj_\ell(\lambda) \ = \ |\m_\sfp\cdots\m_\ell|
\end{equation}
be the length of the sub-partition $\m_\sfp\cdots\m_\ell$.  (In all cases, $\ttj_0 = |\lambda| = \tti$.)  The following is \cite[Proposition 3.30]{MR2960030}.

\begin{lemma}[{\cite{MR2960030}}] \label{L:part}
Let $\lambda = (\lambda_1,\ldots,\lambda_\tti)$ be a partition satisfying \eqref{E:Apart}, and let $\lambda = \m_\sfp\cdots\m_1\m_0$ be the decomposition of $\lambda$ into maximal blocks of consecutive integers.  The pair $\sfa$, $\ttJ = \{\ttj_\sfp,\ldots,\ttj_1,\ttk_1,\ldots,\ttk_\sfq\}$ characterizing the Schubert variety $Y_\lambda$ is given by \eqref{E:jA}, 
$$
  \{ \ttk_1 , \ldots , \ttk_\sfq \} \ = \ 
  \{ \tti - \ttj_\sfp + \lambda_{\ttj_\sfp} \,,\ldots ,\,
  \tti - \ttj_1+\lambda_{\ttj_1} \,,\, \lambda_\tti\} 
  \backslash\{\tti,n+1\} \,,
$$
and
$$
  \sfa \ = \ \left\{ \begin{array}{ll}
    \sfp  & \hbox{ if } \lambda_1 > 1\\
    \sfp-1  & \hbox{ if } \lambda_1 = 1 
  \end{array} \right\} \ = \ \left\{ \begin{array}{ll}
    \sfq \,, & \hbox{ if } \lambda_\tti = n+1 \\
    \sfq-1 \,, & \hbox{ if } \lambda_\tti < n+1 \,.
  \end{array} \right.
$$
Conversely, given $\sfa,\ttJ$, the associated partition $\lambda = \m_\sfp \cdots \m_1\m_0$ is given by 
$$
  \m_\ell \ = \ 
  ( \ttj_{\ell+1} + \ttk_m - \tti + 1 \,,\ldots,\, 
    \ttj_\ell + \ttk_m - \tti) \,,
$$
with $\ell+m = \sfa+1$.
\end{lemma}

\begin{example}
Consider $X = \tGr(5,13) \simeq A_{12}/P_5$.  For the marking $\ttJ = \{ 2,3,7,9,12 \}$ and integer $\sfa=2$, we have $\lambda = (3,4,7,11,12)$.
\end{example}

\subsection{Lagrangian Grassmannians $\tLG(n,2n) = C_n/P_n$} \label{S:aJC}

There exists a bijection between $W^\fp\backslash\{1,w_0\}$ and pairs $\sfa \ge0$ and $\ttJ = \{ \ttj_\sfp , \ldots , \ttj_1 \} \subset \{1,\ldots,n-1\}$ satisfying 
\begin{equation}\label{E:CD_J}
  1 \ \le \ \ttj_\sfp \,<\,\cdots \,<\, \ttj_1 \  \le \ n-1
\end{equation}
and 
$$
  \sfp \ \in \ \{\sfa, \sfa+1\}\,; 
$$
see  \cite[Corollary 3.17]{MR2960030}.  (These $\ttj_\ell$ have the opposite order of those in \cite{MR2960030}.)  For convenience we set
\begin{equation}\label{E:CDendpts}
  \ttj_{\sfp+1} \, := \, 0 \, , \quad \ttj_0 \, := \, n \, .
\end{equation}

It is well-known that Schubert varieties in $X = \tLG(n,2n)$ are indexed by partitions $\lambda = (\lambda_1,\lambda_2,\ldots,\lambda_n)$ such that  
\begin{subequations}\label{SE:Cpart}
\begin{eqnarray}
  & & 1 \le \lambda_1 < \lambda_2 < \cdots < \lambda_n \le 2n \,, \quad \hbox{and} \\
\label{E:Cpart}
  & & \lambda_i \in \lambda \quad\hbox{if and only if} \quad 
  2n+1-\lambda_i\not\in\lambda \,, 
\end{eqnarray} 
\end{subequations}
cf. \cite[\S9.3]{MR1782635}.  The corresponding Schubert variety is given by \eqref{E:Y1}, with $F^\sbullet$ an isotropic flag as in Section \ref{A:not}.

\begin{lemma}[{\cite{flex}}] \label{L:Cpart}
Let $\lambda = (\lambda_1,\ldots,\lambda_n)$ be a partition satisfying \eqref{SE:Cpart}.  Let $\lambda = \m_\sfp\cdots\m_1\m_0$ be a decomposition of $\lambda$ into $\sfp+1$ maximal blocks of consecutive integers.  Then $\ttJ(\lambda) = \{ \ttj_\sfp(\lambda) , \ldots , \ttj_1(\lambda)\}$ is given by \eqref{E:jA}, and 
$$
  \sfa(\lambda) \ = \ \left\{ \begin{array}{ll}
  \sfp - 1 & \hbox{ if } \lambda_1 = 1 \\
  \sfp & \hbox{ if } \lambda_1 > 1 \, .
  \end{array}\right.
$$
Conversely, given $\sfa$ and $\ttJ=\{\ttj_\sfp,\cdots,\ttj_1\}$ we construct $\lambda(\sfa,\ttJ)=\m_\sfp(\sfa,\ttJ)\cdots\m_0(\sfa,\ttJ)$ by
\begin{equation} \label{E:Cblock}
  \m_\ell(\sfa,\ttJ) \ = \
  (n+1+\ttj_{\ell+1}-\ttj_m \,,\, \ldots \,,\, n+\ttj_\ell-\ttj_m ) \,,
\end{equation}
with $\ell+m=\sfa+1$.
\end{lemma}

As an example, Table \ref{t:C5P5} lists the partitions $\lambda$ and corresponding $\sfa:\ttJ$ values for the Schubert varieties in $\tLG(5,10)$.
\begin{small}
\begin{table}[!h]
\caption[Lagrangian Grassmannian $\tLG(5,10)$]{Schubert varieties of $\tLG(5,10)$.}
\label{t:C5P5}
\renewcommand{\arraystretch}{1.2}
\begin{tabular}{|c|c||c|c||c|c|}
\hline
   $\lambda$ & $\sfa:\ttJ$ & $\lambda$ & $\sfa:\ttJ$ 
   & $\lambda$ & $\sfa:\ttJ$ \\ \hline \hline
      $(1,2,3,4,5)$ &         & 
      $(1,2,3,4,6)$ & $0:4$ & 
      $(1,2,3,5,7)$ & $1:3,4$ \\ \hline
      $(1,2,4,5,8)$ & $1:2,4$ & 
      $(1,2,3,6,7)$ & $0:3$ & 
      $(1,3,4,5,9)$ & $1:1,4$ \\ \hline
      $(1,2,4,6,8)$ & $2:2,3,4$ & 
      $(2,3,4,5,10)$ & $1:4$ & 
      $(1,3,4,6,9)$ & $2:1,3,4$ \\ \hline
      $(1,2,5,7,8)$ & $1:2,3$ & 
      $(2,3,4,6,10)$ & $2:3,4$ & 
      $(1,3,5,7,9)$ & $3:1,2,3,4$ \\ \hline
      $(1,2,6,7,8)$ & $0:2$ & 
      $(2,3,5,7,10)$ & $3:2,3,4$ & 
      $(1,4,5,8,9)$ & $1:1,3$ \\ \hline
      $(1,3,6,7,9)$ & $2:1,2,4$ & 
      $(2,4,5,8,10)$ & $3:1,3,4$ & 
      $(2,3,6,7,10)$ & $2:2,4$ \\ \hline
      $(1,4,6,8,9)$ & $2:1,2,3$ & 
      $(3,4,5,9,10)$ & $1:3$ & 
      $(2,4,6,8,10)$ & $4:1,2,3,4$ \\ \hline
      $(1,5,7,8,9)$ & $1:1,2$ & 
      $(3,4,6,9,10)$ & $2:2,3$ & 
      $(2,5,7,8,10)$ & $3:1,2,4$ \\ \hline
      $(1,6,7,8,9)$ & $0:1$ & 
      $(3,5,7,9,10)$ & $3:1,2,3$ & 
      $(2,6,7,8,10)$ & $2:1,4$ \\ \hline
      $(4,5,8,9,10)$ & $1:2$ & 
      $(3,6,7,9,10)$ & $2:1,3$ & 
      $(4,6,8,9,10)$ & $2:1,2$ \\ \hline
      $(5,7,8,9,10)$ & $1:1$ & 
      $(6,7,8,9,10)$ &       &  &  \\
\hline
\end{tabular}
\end{table}
\end{small}
%

\subsection{Spinor varieties $\cS_n = D_n/P_n$} \label{S:aJD}

Given $\sfa = \sfa(w)$ and $\ttJ = \ttJ(w)$, note that 
$$
  \a_{n-1}(Z_w) = 0 \ \hbox{ if } n-1\not\in\ttJ \,,\quad
  \hbox{ and } \quad  \a_{n-1}(Z_w) = 1 \ \hbox{ if } n-1\in\ttJ \,.
$$
Define 
\begin{equation} \label{E:r}
  \sfr \ = \ \left\lceil \half \left(\sfa + \a_{n-1}(Z_w)\right) \right\rceil \ = \ 
  \left\{ \renewcommand{\arraystretch}{1.2} \begin{array}{ll}
    \lceil \sfa/2 \rceil & \hbox{ if } n-1\not\in\ttJ \,,\\
    \lfloor \sfa/2 \rfloor + 1 & \hbox{ if } n-1\in\ttJ \,;
  \end{array} \right.
\end{equation}
There exists a bijection between $W^\fp\backslash\{1,w_0\}$, and pairs $\sfa \ge0$ and $\ttJ = \{ \ttj_\sfp , \ldots , \ttj_1 \} \subset \{1,\ldots,n-1\}$, ordered by \eqref{E:CD_J} and satisfying
\begin{equation} \label{E:Dr}
  \sfp - \a_{n-1}(Z_w) \,\in\, \{ \sfa , \sfa+1 \}\,, \quad\hbox{and}\quad
  2 \le \ttj_\sfr - \ttj_{\sfr+1} \ \hbox{ when } \ \sfr > \a_{n-1}(Z_w) \, ;
\end{equation}
see  \cite[Corollary 3.17]{MR2960030}.  (These $\ttj_\ell$ have the opposite order of those in \cite{MR2960030}.)   We maintain the convention \eqref{E:CDendpts}.

It is well-known that the Schubert varieties of $X=\cS_n$ are indexed by partitions $\lambda = (\lambda_1,\lambda_2,\ldots,\lambda_n)$ satisfying \eqref{SE:Cpart} and 
\begin{equation} \label{E:Dpart}
  \# \{ i \ | \ \lambda_i > n \} \quad\hbox{is even,}
\end{equation}  
cf. \cite[\S9.3]{MR1782635}.  The corresponding Schubert variety is given by \eqref{E:Y1}, with $F^\sbullet$ an isotropic flag as in Section \ref{A:not}.

We define $\ttJ(\lambda)$ as in \eqref{E:jA}, with the following modification of the block decomposition.  In the block decomposition $\lambda = \hat\m_p\cdots\hat\m_1\hat\m_0$, the integers $n-1,n+1$ are considered `consecutive' and are placed in the same $\hat\m_s$--block; likewise, the integers $n,n+2$ are `consecutive.'  For example, if $n=5$, then $\lambda = (2,3,4,6,10)$ has block decomposition $\hat\m_1\hat\m_0 = (2,3,4,6)(10)$; likewise, $\lambda = (1,2,5,7,8)$ has block decomposition $\hat\m_1\hat\m_0 = (1,2)(5,7,8)$.

As before,
\begin{equation} \label{E:JDn}
  \ttj_\ell(\lambda) \ = \ | \hat\m_\sfp\cdots\hat\m_\ell |\,.
\end{equation}
Define
\begin{equation}\label{E:aDn}
  \sfa \ = \ \left\{ \begin{array}{ll}
  \sfp - 2 & 
  \hbox{ if } \lambda_1 = 1 \hbox{ and } \lambda_n - \lambda_{n-1} > 1 \,,\\
  \sfp - 1 & 
  \hbox{ if } \lambda_1 = 1 \hbox{ and } \lambda_n - \lambda_{n-1} = 1 \,,
  \ \hbox{or } \lambda_1 > 1 \hbox{ and } \lambda_n - \lambda_{n-1} > 1 \,,\\
  \sfp & 
  \hbox{ if } \lambda_1 > 1 \hbox{ and }\lambda_n - \lambda_{n-1} = 1  \,.
  \end{array}\right.
\end{equation}

\begin{lemma}[{\cite{flex}}] \label{L:Dpart}
Given a partition $\lambda$ indexing a Schubert variety \eqref{E:Y1} in $\cS_n = D_n/P_n$, the set $\ttJ(\lambda) = \{ \ttj_\sfp(\lambda) , \ldots , \ttj_1(\lambda)\}$ is given by \eqref{E:JDn}, and $\sfa(\lambda)$ is given by \eqref{E:aDn}.  

Conversely, given $\sfa$ and $\ttJ=\{\ttj_\sfp,\cdots,\ttj_1\}$, we construct $\lambda(\sfa,\ttJ)$ as follows.  Let $\lambda' = \m_\sfp\cdots\m_1\m_0$ be given by
\eqref{E:Cblock}, with $\ell+m=\sfa+1+\a_{n-1}(Z_w)$.  If $\lambda'$ satisfies \eqref{E:Dpart}, then $\lambda = \lambda'$.  If \eqref{E:Dpart} fails for $\lambda'$, then we modify the partition as follows: precisely one of $\{ n,n+1\}$ belongs to $\lambda'$, denote this element by $a'$, and the other by $a$.  Then $\lambda$ is obtained from $\lambda'$ by replacing $a'$ with $a$.  
\end{lemma}

As an example, Table \ref{t:D6P6} lists the partitions and corresponding $\sfa:\ttJ$ and $\sfr$ values for the Schubert varieties of $\cS_6 = \tSpin_{12}\bC/P_6$.  
\begin{small}
\begin{table}[!h]
\caption[Spinor variety $\cS_6$]{Schubert varieties of $\cS_6$.}
\label{t:D6P6}
\renewcommand{\arraystretch}{1.2}
\begin{tabular}{|c|c|c||c|c|c|}
\hline
   $\lambda$ & $\sfa:\ttJ$ & $\sfr$ & 
   $\lambda$ & $\sfa:\ttJ$ & $\sfr$ \\ \hline \hline
      $(1,2,3,4,5,6)$ & &  &
      $(1,2,3,4,7,8)$ & $0:4$ & $0$ \\ \hline
      $(1,2,3,5,7,9)$ & $0:3,5$ & $1$ & 
      $(1,2,4,5,7,10)$ & $0:2,5$ & $1$ \\ \hline
      $(1,2,3,6,8,9)$ & $0:3$ & $0$ & 
      $(1,3,4,5,7,11)$ & $0:1,5$ & $1$ \\ \hline
      $(1,2,4,6,8,10)$ & $1:2,3,5$ & $1$ & 
      $(2,3,4,5,7,12)$ & $0:5$ & $1$ \\ \hline
      $(1,3,4,6,8,11)$ & $1:1,3,5$ & $1$ & 
      $(1,2,5,6,9,10)$ & $1:2,4$ & $1$ \\ \hline
      $(2,3,4,6,8,12)$ & $1:3,5$ & $1$ & 
      $(1,3,5,6,9,11)$ & $2:1,2,4,5$ & $2$ \\ \hline
      $(1,2,7,8,9,10)$ & $0:2$ & $0$ & 
      $(2,3,5,6,9,12)$ & $2:2,4,5$ & $2$ \\ \hline
      $(1,4,5,6,10,11)$ & $1:1,4$ & $1$ & 
      $(1,3,7,8,9,11)$ & $1:1,2,5$ & $1$ \\ \hline
      $(2,4,5,6,10,12)$ & $2:1,4,5$ & $2$ & 
      $(2,3,7,8,9,12)$ & $1:2,5$ & $1$ \\ \hline
      $(1,4,7,8,10,11)$ & $2:1,2,4$ & $1$ & 
      $(3,4,5,6,11,12)$ & $1:4$ & $1$ \\ \hline
      $(2,4,7,8,10,12)$ & $3:1,2,4,5$ & $2$ & 
      $(1,5,7,9,10,11)$ & $1:1,3$ & $1$ \\ \hline
      $(3,4,7,8,11,12)$ & $2:2,4$ & $1$ & 
      $(2,5,7,9,10,12)$ & $2:1,3,5$ & $2$ \\ \hline
      $(1,6,8,9,10,11)$ & $0:1$ & $0$ & 
      $(3,5,7,9,11,12)$ & $3:1,3,4$ & $2$ \\ \hline
      $(2,6,8,9,10,12)$ & $1:1,5$ & $1$ & 
      $(4,5,7,10,11,12)$ & $1:3$ & $1$ \\ \hline
      $(3,6,8,9,11,12)$ & $2:1,4$ & $1$ & 
      $(4,6,8,10,11,12)$ & $2:1,3$ & $1$ \\ \hline
      $(5,6,9,10,11,12)$ & $1:2$ & $1$ & 
      $(7,8,9,10,11,12)$ & & \\ 
\hline
\end{tabular}
\end{table}
\end{small}


\section{The exceptional cases} \label{S:aJE}

Figures \ref{f:E6} and \ref{f:E7} (pages \pageref{f:E6} and \pageref{f:E7}) are respectively the Hasse diagrams $W^\fp$ of the Cayley plane $E_6/P_6$ and Freudenthal variety $E_7/P_7$.  Each node represents a Schubert class $\xi_w = [Y_w]$ and is labeled with the corresponding $\sfa(w):\ttJ(w)$ values, which we obtained with the assistance of \cite{LiE}.  The height of the node indicates the dimension of $Y_w$; in particular, the lowest node $o \in X$ is at height zero.  Two nodes are connected if the Schubert variety associated with the lower node is a divisor of the Schubert variety associated with the higher node.

\begin{figure}[htb]
\setlength{\unitlength}{35pt}
\begin{picture}(5,16)
\put(5,0){\line(-1,1){5}}\put(3,4){\line(-1,1){2}}
\put(3,6){\line(-1,1){2}}\put(4,7){\line(-1,1){4}}
\put(5,8){\line(-1,1){4}}\put(3,12){\line(-1,1){1}}
\put(0,11){\line(1,1){5}}\put(1,10){\line(1,1){2}}
\put(1,8){\line(1,1){2}}\put(0,5){\line(1,1){4}}
\put(1,4){\line(1,1){4}}\put(2,3){\line(1,1){1}}
\multiput(5,0)(-1,1){6}{\circle*{0.15}}
\multiput(3,4)(-1,1){3}{\circle*{0.15}}
\multiput(3,6)(-1,1){3}{\circle*{0.15}}
\multiput(4,7)(-1,1){5}{\circle*{0.15}}
\multiput(5,8)(-1,1){5}{\circle*{0.15}}
\multiput(3,12)(-1,1){2}{\circle*{0.15}}
\multiput(5,16)(-1,-1){3}{\circle*{0.15}}
\thinlines
\put(4.2,14.9){1:2} \put(3.2,13.9){2:4} \put(2.2,12.9){3:35}
\put(3.2,11.9){1:3} \put(1.2,11.9){2:15}
\put(0.2,10.9){1:5} \put(2.2,10.9){3:135}
\put(1.2,9.9){2:35} \put(3.2,9.9){2:14}
\put(2.2,8.9){3:145} \put(4.2,8.9){1:12}
\put(1.2,7.9){1:4} \put(3.2,7.9){2:125} 
\put(5.2,7.9){0:1} \put(5.1,8.25){$Q^8$} 
\put(2.2,6.9){2:124} \put(4.2,6.9){1:15}
\put(1.2,5.9){1:23} \put(3.2,5.9){1:14}
\put(0.2,4.9){0:2} \put(-0.1,5.25){$\bP^5$} 
\put(2.2,4.9){1:123}
\put(1.2,3.9){0:12} \put(3.2,3.9){0:3} \put(3.2,4.2){$\bP^4$} 
\put(2.2,2.9){0:23} \put(3.2,1.9){0:4} \put(4.2,0.9){0:5}
\normalsize
\put(5.2,-0.1){$o \in X$} \put(5.2,15.9){$X$}
\end{picture}
\caption[The Hasse diagram for the Cayley plane $E_6/P_6$]{Hasse diagram of $E_6/P_6$, each node labeled with the $\sfa:\ttJ$ values.}
\label{f:E6}
\addcontentsline{toc}{section}{Figure \ref{f:E6}: Hasse poset and $\sfa,\ttJ$ data for the Cayley plane $E_6/P_6$}
\end{figure}

\begin{figure}[htb]
\setlength{\unitlength}{20pt}
\begin{picture}(8,27)
\put(2,6){\line(1,-1){6}} \put(3,7){\line(1,-1){2}}
\put(3,9){\line(1,-1){2}} \put(2,12){\line(1,-1){4}}
\put(3,13){\line(1,-1){4}} \put(4,14){\line(1,-1){4}}
\put(0,17){\line(1,-1){4}} \put(1,18){\line(1,-1){4}}
\put(2,19){\line(1,-1){4}} \put(3,20){\line(1,-1){2}}
\put(3,22){\line(1,-1){2}} \put(0,27){\line(1,-1){6}}
\put(4,4){\line(1,1){1}} \put(3,5){\line(1,1){5}}
\put(2,6){\line(1,1){5}} \put(3,9){\line(1,1){3}}
\put(3,11){\line(1,1){2}} \put(2,12){\line(1,1){2}}
\put(4,13){\line(1,1){2}} \put(3,14){\line(1,1){2}}
\put(4,12){\line(0,1){1}} \put(4,14){\line(0,1){1}}
\put(3,13){\line(0,1){1}} \put(5,13){\line(0,1){1}}
\put(2,15){\line(1,1){3}} \put(1,16){\line(1,1){5}}
\put(0,17){\line(1,1){5}} \put(3,22){\line(1,1){1}}
\multiput(2,6)(1,-1){7}{\circle*{0.25}} \multiput(3,7)(1,-1){3}{\circle*{0.25}}
\multiput(3,9)(1,-1){3}{\circle*{0.25}} \multiput(2,12)(1,-1){5}{\circle*{0.25}}
\multiput(3,13)(1,-1){5}{\circle*{0.25}} \multiput(4,14)(1,-1){5}{\circle*{0.25}}
\multiput(0,17)(1,-1){5}{\circle*{0.25}} \multiput(1,18)(1,-1){5}{\circle*{0.25}}
\multiput(2,19)(1,-1){5}{\circle*{0.25}} \multiput(3,20)(1,-1){3}{\circle*{0.25}}
\multiput(3,22)(1,-1){3}{\circle*{0.25}} \multiput(0,27)(1,-1){7}{\circle*{0.25}}
\thinlines
\footnotesize
\put(8.2,-0.2){$o \in X$} \put(-0.7,26.9){$X$}
\put(5.4,5.2){$\bP^5$} \put(1.3,6.2){$\bP^6$}
\put(8.4,10.2){$Q^{10}$}
\put(7.2,0.85){0:6} \put(6.2,1.85){0:5} \put(5.2,2.85){0:4}
\put(4.2,3.85){0:23} \put(5.4,4.8){0:3} \put(3.3,4.85){0:12}
\put(4.2,5.85){1:123} \put(2.4,5.85){0:2}
\put(5.2,6.85){1:14} \put(3.2,6.85){1:23}
\put(4.2,7.85){2:124} \put(6.2,7.85){1:15}
\put(3.4,8.85){1:4} \put(5.3,8.85){2:125} \put(7.2,8.85){1:16}
\put(4.2,9.85){3:145} \put(6.2,9.85){2:126} \put(8.4,9.75){0:1}
\put(3.2,10.85){2:35} \put(5.2,10.85){3:146} \put(7.2,10.85){1:12}
\put(2.4,11.85){1:5} \put(4.2,11.85){4:1356} \put(6.2,11.85){2:14}
\put(1.6,12.85){3:156} \put(3.6,12.3){2:36} \put(5.2,12.85){3:135}
\put(1.6,13.85){4:356} \put(3.6,14.35){2:15} \put(5.2,13.85){3:136}
\put(2.2,14.85){3:46} \put(4.2,14.85){5:1356} \put(6.4,14.85){1:3}
\put(1.2,15.85){2:26} \put(3.2,15.85){4:146} \put(5.2,15.85){3:35}
\put(0.4,16.85){1:6} \put(2.2,16.85){3:126} \put(4.2,16.85){5:346}
\put(1.2,17.85){2:16} \put(3.3,17.85){4:236} \put(5.2,17.85){2:4}
\put(2.2,18.85){3:36} \put(4.2,18.85){5:246} 
\put(3.2,19.85){4:46} \put(5.2,19.85){3:25}
\put(4.2,20.85){5:256} \put(6.4,20.85){1:2}
\put(3.4,21.85){2:5} \put(5.3,21.85){3:26}
\put(4.2,22.85){4:25} \put(3.2,23.85){3:4} \put(2.2,24.85){2:3}
\put(1.2,25.85){1:1}
\normalsize
\end{picture}
\caption[The Hasse diagram for $E_7/P_7$]{Hasse diagram of $E_7/P_7$, each node labeled with the $\sfa:\ttJ$ values.}
\label{f:E7}
\addcontentsline{toc}{section}{Figure \ref{f:E7}: Hasse poset and $\sfa,\ttJ$ data for the Freudenthal variety $E_7/E_7$}
\end{figure}

\bibliographystyle{plain}
\def\cprime{$'$}

\end{document}